\documentclass[11pt,letterpaper]{amsart}

\usepackage{verbatim, latexsym, amssymb, amsmath,color, mathrsfs}
\usepackage{enumitem}
\usepackage{epsfig}
\usepackage{xcolor}

\usepackage{stmaryrd}

\usepackage{geometry}
\usepackage{hyperref}

\DeclareMathOperator{\End}{End}

\DeclareMathOperator{\Ric}{Ric}

\DeclareMathOperator{\spt}{spt}
\DeclareMathOperator{\Id}{Id}
\DeclareMathOperator{\diam}{diam}

\DeclareMathOperator{\Vol}{Vol}
\DeclareMathOperator{\length}{length}
\DeclareMathOperator{\Area}{Area}

\DeclareMathOperator{\dist}{dist}

\DeclareMathOperator{\Jac}{Jac}
\DeclareMathOperator{\spherevol}{SphereVol}

\DeclareMathOperator{\id}{id}
\DeclareMathOperator{\set}{set}

\DeclareMathOperator{\dvol}{dvol}

\newtheorem{theo}{Theorem}[section]
\newtheorem{prop}[theo]{Proposition}
\newtheorem{lemme}[theo]{Lemma}
\newtheorem{definition}[theo]{Definition}
\newtheorem{coro}[theo]{Corollary}

\theoremstyle{definition}
\newtheorem{remarque}[theo]{Remark}
\newtheorem{exemple}[theo]{Example}
\newtheorem{assumption}[theo]{Assumption}

\begin{document}
\title[Entropy and stability of hyperbolic manifolds]
{Entropy and stability of hyperbolic manifolds}

\author{Antoine Song}
\address{California Institute of Technology\\ 177 Linde Hall, \#1200 E. California Blvd., Pasadena, CA 91125}
\email{aysong@caltech.edu}

\maketitle

\begin{abstract} 

Let $(M,g_0)$ be a closed oriented hyperbolic manifold of dimension at least $3$. 
By the volume entropy inequality of G. Besson, G. Courtois and S. Gallot, for any Riemannian metric  $g$ on $M$ with same volume as $g_0$, its volume entropy $h(g)$ satisfies
$h(g)\geq n-1$ with equality only when $g$ is isometric to $g_0$. We show that the hyperbolic metric $g_0$ is stable in the following sense:
if $g_i$ is a sequence of Riemaniann metrics on $M$ of same volume as $g_0$ and if $h(g_i)$ converges to $n-1$, then there are smooth subsets $Z_i\subset M$  such that both $\Vol(Z_i,g_i)$ and $\Area(\partial Z_i,g_i)$ tend to $0$, and $(M\setminus Z_i,g_i)$ 
converges to $(M,g_0)$ in the measured Gromov-Hausdorff topology.
The proof relies on showing that any spherical Plateau solution for $M$ is intrinsically isomorphic to $(M,\frac{(n-1)^2}{4n} g_0)$.


\end{abstract}


\section*{Introduction}

Let $M$ be a hyperbolic manifold of dimension at least $3$ with hyperbolic metric $g_0$. If $g$ is a Riemannian metric on $M$, let $h(g)$ denote its volume entropy:
$$h(g):=\lim_{R\to \infty} \frac{\log \Vol(\tilde{B}_g(o,R),g)}{R}$$
where $\tilde{B}_g(o,R)$ denotes the geodesic $R$-ball centered at some point $o$ in the universal cover $(\tilde{M},g)$ of $(M,g)$. 
The fundamental volume entropy inequality, proved by Besson-Courtois-Gallot in \cite{BCG95,BCG96}, asserts that for any Riemannian metric $g$ on $M$ of same volume as $g_0$, we have
\begin{equation}\label{entropy inequality}
h(g)\geq h(g_0)=n-1.
\end{equation}
Moreover, Besson-Courtois-Gallot showed  that this inequality is rigid in the sense that if equality holds, then $g$ is isometric to $g_0$. How stable is the volume entropy inequality?
We find that stability holds after removing negligible subsets:

\begin{theo} \label{theorem:stable}
Let $(M,g_0)$ be a closed oriented hyperbolic manifold of dimension at least $3$. Let $\{g_i\}_{i\geq1}$ be a sequence of Riemannian metrics on $M$ with $\Vol(M,g_i) = \Vol(M,g_0)$. If 
$$\lim_{i\to  \infty} h(g_i) = n-1,$$
then there is a sequence of smooth subsets $Z_i \subset M$ such that 
$$\lim_{i\to \infty}\Vol(Z_i,g_i) = \lim_{i\to \infty}\Area(\partial Z_i,g_i) =0$$ and 
$(M\setminus Z_i,g_i)$ converges to $(M,g_0)$ in the measured Gromov-Hausdorff topology.

\end{theo}

\vspace{1em}

In the statement of Theorem \ref{theorem:stable}, $(M\setminus Z_i,g_i)$ is the metric space where the distance between two points $a,b\in M\setminus Z_i$ is given by the infimum of the $g_i$-lengths of curves joining $a$ to $b$ inside $ M\setminus Z_i$. 
A sequence of manifolds converges in the measured Gromov-Hausdorff topology if it converges both in the Gromov-Hausdorff and Gromov-Prokhorov topologies (for a definition of those topologies, see \cite[Chapter 27, page 778]{Vil09}).
Gromov-Prokhorov convergence implies $\lim_{i\to \infty}\Vol(Z_i,g_i) =0$. On the other hand, the conclusion that $\lim_{i\to \infty}\Area(\partial Z_i,g_i) =0$ is a strong additional  property.

It is elementary to see that naive stability for the Gromov-Hausdorff topology does not hold. Indeed, by adding thin and long threads to the hyperbolic metric $g_0$, we get a new metric $g$ whose volume and volume entropy are arbitrarily close to $\Vol(M,g_0)$ and $n-1$ respectively. In this example, $(M,g)$ is far from $(M,g_0)$ in the Gromov-Hausdorff topology, although it is still close to $(M,g_0)$ in the Gromov-Prokhorov topology. 
The following question remains open: \emph{under the assumptions of Theorem \ref{theorem:stable}, does $(M,g_i)$ converge to $(M,g_0)$ in the Gromov-Prokhorov topology?} In Remark \ref{optimal}, we discuss the optimality of Theorem \ref{theorem:stable} with a notion of ``coarse dimension'' for Riemannian manifolds.

\vspace{1em}

\subsection*{Historical comments}

The question of stability for the volume entropy was raised by Courtois in \cite{Courtois98}, and variants of this problem have been previously studied
by Bessi\`{e}res-Besson-Courtois-Gallot \cite{BBCG12} under a lower bound on the Ricci curvature (see also \cite{LW11}), by 
Guillarmou-Lefeuvre \cite{GL19} and Guillarmou-Knieper-Lefeuvre \cite{GKL22} for neighborhoods of negatively curved manifolds, and Butt \cite{Butt22} assuming uniform negative curvature bounds. 
We note that the differential rigidity result of \cite{BBCG12} should follow from Theorem \ref{theorem:stable} and the theory of Cheeger-Colding \cite[Theorem A.1.12]{CC97}.

The stability of geometric inequalities for Riemannian manifolds is a theme that has been extensively studied. We emphasize that in Theorem \ref{theorem:stable}, no a priori curvature bound is required. 
The proof of this result has thus a quite different flavor compared to stability results under curvature bounds. 
Theorem \ref{theorem:stable} provides a stability result  after removing ``negligible'' subset. This is formally similar to a stability result we recently proved with Conghan Dong  for the Positive Mass theorem \cite{DS23}, which settles a conjecture of Huisken-Ilmanen. 
For stability results in the context of curvature bounds, see \cite{Colding96a,Colding96b,Petersen99,Aubry05,CRX19,CDNPSW21}... for Ricci curvature, see \cite{LS14,HLS17,Sormani21, LNN20,Allen21,CL22,DS23}... for scalar curvature.  For spectral isoperimetric inequalities on surfaces, see \cite{KNPS21} and references therein. 



\vspace{1em}

   \subsection*{Main ingredients}

The first main input in the proof of Theorem \ref{theorem:stable} is the theory of integral currents in metric spaces from geometric measure theory \cite{AK00,Lang11,Wenger11,SW11} 
In particular, we make essential use of a compactness theorem due to Wenger \cite{Wenger11} which is formulated in terms of the intrinsic flat topology for integral current spaces \cite{SW11}.  
With some hindsight, revisiting Besson-Courtois-Gallot's original work using tools from geometric measure theory is especially natural, which is  one of the main points of this paper. For instance, this combination leads directly to the ``spherical Plateau problem'' described in the next subsection, which enjoys rigidity properties at least as strong as for the minimal volume entropy problem.

The second ingredient in the proof of Theorem \ref{theorem:stable} is a sharp comparison result for the volume entropy of manifolds almost metrically dominated by a closed hyperbolic manifold, Theorem \ref{cle}. Its proof relies on the equidistribution of geodesic spheres in closed hyperbolic manifolds. As a side note, together with Demetre Kazaras and Kai Xu, we recently applied this comparison result together with a ``drawstring'' construction to provide counterexamples to a conjecture of Agol-Storm-Thurston relating scalar curvature and volume entropy \cite{KSX23}.

These two parts together yield a stability result stronger than Theorem \ref{theorem:stable}: under the same assumptions, $(M\setminus Z_i,g_i)$ actually converges to $(M,g_0)$ with respect to the intrinsic flat topology, see Theorem \ref{intrinsic flat stability}. 

\vspace{1em}

   \subsection*{The spherical Plateau problem}

The proof of Theorem \ref{theorem:stable} is closely related to a  variational problem in infinite dimension, called the spherical Plateau problem. Let $(M,g_0)$ be a closed oriented hyperbolic manifold with its hyperbolic metric and let $\Gamma:=\pi_1(M)$.
Consider the unit sphere $S^\infty$ in the Hilbert space $\ell^2(\Gamma)$ and let $\Gamma$ act on $S^\infty$ by the regular representation $\lambda_\Gamma$. Denote by $S^\infty/\lambda_{\Gamma}(\Gamma)$ the corresponding quotient manifold, endowed with the standard round Hilbert Riemannian metric $\mathbf{g}_{\text{Hil}}$. There is a unique homotopy class $\mathscr{H}_M$ of smooth immersions from $M$ to $S^\infty/\lambda_{\Gamma}(\Gamma)$ inducing an isomorphism on the fundamental groups. Besson-Courtois-Gallot define the \emph{spherical volume} of $M$ \cite{BCG91} as follows
$$\spherevol(M) := \inf\{\Vol(M,\phi^*\mathbf{g}_{\text{Hil}} ); \quad \phi \in \mathscr{H}_M\}.$$

A key step in Besson-Courtois-Gallot's  proof \cite{BCG95,BCG96} of the entropy inequality (\ref{entropy inequality}) is to establish that
$$\spherevol(M)=\Vol(M, \frac{(n-1)^2}{4n} g_0).$$ 
That result led us to consider  in \cite{Antoine23a} the corresponding volume minimization problem, in particular  the study of ``limits of minimizing sequences''. Consider  any minimizing sequence of maps $\phi_i\in \mathscr{H}_M$, namely a sequence such that 
$$\lim_{i\to \infty}\Vol(M,\phi_i^*\mathbf{g}_{\text{Hil}}) = \spherevol(M).$$
Then by Wenger's compactness theorem \cite{Wenger11}, the images $\phi_i(M)$ subsequentially converge as integral current spaces to an integral current space 
$$C_\infty = (X_\infty,d_\infty, T_\infty)$$ in the intrinsic flat topology, in the sense of Sormani-Wenger \cite{SW11}. Here $(X_\infty,d_\infty)$ is a metric space, $T_\infty$ is an integral current in the completion of $(X_\infty,d_\infty)$, see Subsection \ref{appendix a}. We call any such limit $C_\infty$ a \emph{spherical Plateau solution} for $M$.

Our second main theorem concerns the intrinsic uniqueness  of spherical Plateau solutions for hyperbolic manifolds. The notion of ``intrinsic isomorphism'' between two integral current spaces will be defined in Definition \ref{definition:intrinsic isomorphism}.

\begin{theo}\label{theorem:unique}
If $(M,g_0)$ is a closed oriented hyperbolic manifold of dimension $n\geq 3$, then any spherical Plateau solution for $M$ is intrinsically isomorphic to $(M,\frac{(n-1)^2}{4n} g_0)$.
\end{theo}

Theorem \ref{theorem:unique} leads to a rigidity result with a representation theoretic flavor for $\pi_1(M)$, see \cite[Corollary 4.3]{Antoine23a}.
Conjecturally, the spherical Plateau solution for a closed oriented hyperbolic manifold is unique  \cite[Question 8]{Antoine23a}. 
The spherical Plateau problem is of independent geometric interest: in \cite{Antoine23a}, we sketch the proof of the intrinsic uniqueness of spherical Plateau solutions for all oriented closed $3$-manifolds, and the construction of higher dimensional analogues of hyperbolic Dehn fillings. 
Strictly speaking, the statement of Theorem \ref{theorem:unique} is not necessary to show Theorem \ref{theorem:stable}. However, the methods in its proof do play a central role.


\begin{remarque} 
The arguments in this paper extend to closed  oriented manifolds which are locally symmetric of rank one due to \cite{BCG96,Ruan22}, and so versions of the main theorems hold more generally for these spaces.
\end{remarque}

\vspace{1em}

\subsection*{Outline of the proofs}
{}\

\textbf{For Theorem} \ref{theorem:unique}:
In order to describe the proof, it is helpful to recall how  Besson-Courtois-Gallot were able to compute the spherical volume 
$$\spherevol(M)=\Vol(M, \frac{(n-1)^2}{4n} g_0).$$ Their main tool was the barycenter map ${\mathrm{Bar}}$. In our setting, this is a Lipschitz map which under some technical conditions sends cycles in $S^\infty/\lambda_{\Gamma}(\Gamma)$ of the form $\phi(M)$, where $\phi\in \mathcal{F}(M)$, to the rescaled hyperbolic manifold $(M, \frac{(n-1)^2}{4n} g_0)$ with topological degree $1$. Roughly speaking, the Jacobian of restriction of the barycenter map ${\mathrm{Bar}}: \phi(M) \to (M, \frac{(n-1)^2}{4n} g_0)$ satisfies \cite{BCG95,BCG96} 
\begin{equation} \label{bound1}
|\Jac {\mathrm{Bar}}| \leq 1,
\end{equation}
which in particular implies that $\spherevol(M)\geq \Vol(M, \frac{(n-1)^2}{4n} g_0)$. Then the opposite inequality is checked by finding an explicit sequence of embeddings $\phi_i\in \mathcal{F}(M)$ such that $\lim_{i\to \infty}\Vol(M,\phi_i^*\mathbf{g}_{\text{Hil}}) = \Vol(M, \frac{(n-1)^2}{4n} g_0).$

In order to show that spherical Plateau solutions are unique up to intrinsic isomorphism, we try to argue as follows. 
Consider a minimizing sequence of maps $\phi_i\in \mathcal{F}(M)$, and denote by $C_i$ the integral currents of $S^\infty/\lambda_{\Gamma}(\Gamma)$ induced by pushing forward the fundamental class of $M$ by $\phi_i$. The barycenter map ${\mathrm{Bar}}$ enjoys the Jacobian bound (\ref{bound1}) which is almost achieved on a region $\Omega_i \subset \spt(C_i)$ that covers almost all of $\spt(C_i)$ as $i\to \infty$. Nontrivially, this implies a local Lipschitz bound for ${\mathrm{Bar}}$, which holds on a whole neighborhood of $\Omega_i$, and the differential of ${\mathrm{Bar}}$ at points of $\Omega_i$ can be shown to be close to a linear isometry. 
We can assume, by Wenger's compactness theorem, that $C_i$ converges to a spherical Plateau solution 
$$C_\infty =(X_\infty,d_\infty,S_\infty)$$
(the fact that such a limit exists is crucial). 
We then construct a limit map from the support of $S_\infty$ to $M$:
$${\mathrm{Bar}}_\infty : \spt S_\infty \to M$$ sending the current structure $S_\infty$ to the natural current structure $\llbracket 1_M\rrbracket$ supported on $M$.
Heuristically, as $i$ goes to infinity, the Jacobian bound (\ref{bound1}) for ${\mathrm{Bar}} : \spt(C_i) \to (M, \frac{(n-1)^2}{4n} g_0)$ should be almost saturated almost everywhere, which means that the differential of ${\mathrm{Bar}} $ should be close to a linear isometry almost everywhere. 
In other words, ${\mathrm{Bar}} : \spt(C_i) \to (M, \frac{(n-1)^2}{4n} g_0)$  are almost Riemannian isometries. Passing to the limit, we should be able to deduce that ${\mathrm{Bar}}_\infty$ is an isometry for the intrinsic metrics, which  would essentially conclude the proof. 
This strategy of constructing a limit barycenter map has been exploited in the rigidity theorems of \cite[Proposition 7.1]{BCG95} and \cite{BBCG12} where curvature bounds are assumed. 
There, the authors can argue that since their limit barycenter map is $1$-Lipschitz and preserves the volume, it has to be an isometry, see \cite[Proposition C.1]{BCG95} and \cite[Sections 3, 4, 5]{BBCG12}. Related or more general ``Lipschitz-volume'' rigidity results were obtained in \cite[Theorem 1.1]{DNP23}, \cite[Theorem 1.1]{GCS23} and \cite[Theorem 1.2]{Zus23}.

However, all those results depend either on the regularity of the convergence to the limit space outside of a small singular set, or on the $1$-Lipschitz continuity of the limit map. The new challenge in our case is the lack of a priori regularity for spherical Plateau solutions and the fact that the limit map is never $1$-Lipschitz in our situation (even though it will a posteriori follow that it is $1$-Lipschitz for the intrinsic metric on $\spt S_\infty$). 
To address this issue, we show in Proposition \ref{proposition:limit map} that under some natural assumptions, limits of almost Riemannian isometries are Riemannian isometries. The proof  uses a ``curve lifting'' argument, which in turn is based on an averaging argument involving the coarea formula.




\textbf{For Theorem} \ref{theorem:stable}:

Consider a Riemannian metric $g$ on $M$ with same volume as $g_0$ and with entropy close to $n-1$. Then, there is a uniformly Lipschitz map 
$$\mathcal{P} : (M,\frac{(n-1)^2}{4n} g) \to (S^\infty/\lambda_{\Gamma}(\Gamma),\mathbf{g}_{\mathrm{Hil}})$$ which is almost a Riemannian isometry to its image, as observed by Besson-Courtois-Gallot \cite{BCG91}.
We apply again Proposition \ref{proposition:limit map} as in the proof of Theorem \ref{theorem:unique} to $\mathrm{Bar} \circ \mathcal{P}$ instead of $\mathrm{Bar}$.  We deduce that, for smooth subset $Z\subset M$,
\begin{itemize}
\item $\Vol(Z,g)$ and $\Area(\partial Z,g)$ are both small,
\item $(M\setminus Z,g)$ is close in the intrinsic flat topology to a space $C_\infty=(X_\infty,d_\infty,S_\infty)$,
\item there is a bi-Lipschitz, $1$-Lipschitz map 
$$\Psi: (M,g_0) \to (\spt S_\infty,d_\infty),$$
\item $(M\setminus Z,g)$ is  Gromov-Hausdorff close to $(\spt S_\infty,d_\infty)$ via a topologically natural map.
\end{itemize}
The properties of $\mathrm{Bar} \circ \Psi$  are not as good as those of $\mathrm{Bar}$, so unlike Theorem \ref{theorem:unique}, we cannot readily conclude that $\Psi$ is an isometry for the intrinsic metrics. We need to remove a small subset $Z$ from $M$ to get the Gromov-Hausdorff closeness property above.

In order to prove that the map $\Psi$ above is, in fact, an isometry, we rely on a volume entropy comparison result, Theorem \ref{cle}. The latter roughly says that 
if $(M\setminus Z,g)$ is naturally Gromov-Hausdorff close to a metric space $(M,d)$ and if there is a $1$-Lipschitz map $\Psi$ from $(M,g_0)$ to $(M,d)$, then either $\Psi$ is an isometry or the volume entropy of $(M,g)$ is strictly  larger than $n-1$. To show this, we make use of the equidistribution of geodesic spheres in the unit tangent bundles of closed hyperbolic manifolds, a result  shown by Eskin-McMullen in \cite{EskinMcMullen93}.

Applying that comparison result to $(\spt S_\infty, d_\infty)$, we conclude that the map $\Psi$ above is an isometry. This yields the intrinsic flat stability result, Theorem \ref{intrinsic flat stability}. We conclude the proof of Theorem \ref{theorem:stable} by applying a lemma of Portegies \cite{Portegies15}: if a sequence of integral current spaces converges to a limit in the intrinsic flat topology without volume loss, then viewed as metric measure spaces the sequence converges to the limit in the Gromov-Prokhorov topology.

\vspace{1em}

\subsection*{Organisation}

Section \ref{prelimm} is about integral currents in metric spaces and maps between them. We prove a proposition answering in some cases  the following question: given a sequence of uniformly Lipschitz, almost Riemannian isometries converging to a limit map, what can we say about that limit map?

In Section \ref{section:Plateau}, we define the spherical Plateau problem for a closed oriented hyperbolic manifold. We introduce the barycenter map of Besson-Courtois-Gallot in our setting. 
Then we prove the intrinsic uniqueness of spherical Plateau solutions in Theorem \ref{theorem:unique}. 

In Section \ref{section:stability}, we show a technical theorem whose proof is closely related to that of Theorem \ref{theorem:unique}.
We  review an equidistribution result for geodesic spheres in the unit tangent bundle of hyperbolic manifolds, and how it implies a sharp comparaison theorem. Then, we apply the comparison theorem and the technical theorem to establish the volume entropy stability in terms of the intrinsic flat topology, which implies Theorem \ref{theorem:stable}.

\subsection*{Acknowledgements}
I am grateful to G\'{e}rard Besson, Gilles Courtois, Juan Souto, John Lott, Ursula Hamenst\"{a}dt, Ben Lowe and Demetre Kazaras for insightful discussions during the writing of this article. I would especially like to thank Cosmin Manea, Hyun Chul Jang,  Xingzhe Li and Dongming (Merrick) Hua for their careful reading, suggestions and for several corrections. 

A.S. was partially supported by NSF grant DMS-2104254. This research was conducted during the period A.S. served as a Clay Research Fellow.

\section{Limits of currents and limits of almost Riemannian isometries} \label{prelimm}

\subsection{Currents in metric spaces and Wenger's compactness theorem} \label{appendix a}

The theory of currents in metric spaces begins with works of De Giorgi, and Ambrosio-Kirchheim \cite{AK00}. It extends the theory of currents in finite dimensional manifolds due to De Giorgi, Federer-Fleming. For the most part, in this paper we will only stay in the standard framework of smooth maps and smooth manifolds. Nevertheless, a key reason for caring about metric currents is that this general theory enables to formulate powerful compactness results like  Theorem \ref{key compact} below. Besides, there is a profusion of standard tools (weak convergence, area/coarea formulae, slicing, push-forward...) for which the most natural language is given by geometric measure theory. 

The main references we will need on the theory of metric currents are \cite{AK00,AK00b,Wenger11,SW11}.
We reviewed in some details the main definitions and results of the theory in Section 1 of \cite{Antoine23a}. 
In this paper, metrics on metric spaces assume only finite values. Integral currents in complete metric spaces are, roughly speaking, a countable union of Lipschitz push-forwards of Borel subsets  in Euclidean spaces. They give a workable notion of ``generalized oriented submanifolds'' in complete metric spaces like Hilbert manifolds or Banach spaces. An $n$-dimensional  integral current $S$ has a well-defined notion of boundary $\partial S$ which is an $(n-1)$-dimensional  integral current, a notion of volume measure denoted by $\|S\|$ and a notion of total volume called mass $\mathbf{M}(S)$. Such a current $S$ is concentrated on a so-called canonical set $\set(S)$, itself included in the support $\spt(S)$ of the measure $\|S\|$. The restriction of $S$ to a Borel set $A$ is denoted by $S\llcorner A$, and its push-forward by a Lipschitz map $\phi$ is called $\phi_\sharp S$. With those notations, $\mathbf{M}(S\llcorner A) = \|S\|(A).$
See \cite[Section 3]{AK00}, see also \cite[Subsections 1.1, 1.2]{Antoine23a} for a review.

The space of integral currents in a given complete metric space is endowed with the weak topology and flat topology, and the latter is finer than the former, see \cite[Subsection 1.1]{Wenger07} \cite[Subsection 1.3]{Antoine23a}. The mass is lower semicontinuous with respect to convergence in those topologies \cite{AK00}.

The area formula expresses the mass of an integral current by its image under a Lipschitz map \cite[Section 8]{AK00b}, \cite[Section 9]{AK00}, \cite[Subsection 1.4]{Antoine23a}. The coarea formula, a kind of dual formula, expresses the mass of an integral current in terms of a double integral involving level sets of a Lipschitz map \cite[Section 9]{AK00b}, \cite[Subsection 1.4]{Antoine23a}. The slicing theorem is a kind of generalization of Sard's theorem and tells us that almost all level sets of a Lipschitz map are integral currents themselves \cite[Theorems 5.6 and 5.7]{AK00}.

Following  the notion of integral currents in complete metric spaces, 
one can define a more intrinsic notion of integral currents. That was achieved by Sormani-Wenger \cite{SW11}. Basically an integral current space is a triple $(X,d,S)$ where $(X,d)$ is a metric space and $S$ is an integral current in the completion of $(X,d)$, which we will usually denote by $\spt S$ (one requires that $X$ is the ``canonical set'' of the current $S$) \cite[Definition 1.3 and Subsection 1.1]{Antoine23a}.  
A simple example of integral current space is given by a closed, connected, oriented Riemannian $n$-manifold $(N,h)$: the metric space is $N$ endowed with the geodesic distance of $h$, and the integral current structure $\llbracket 1_N \rrbracket$  is the natural integral current induced by the fundamental class $[N]\in H_n(N;\mathbb{Z})$.

There is  also an intrinsic notion of flat topology, called intrinsic flat topology \cite{SW11}. Similarly to the definition of Gromov-Hausdorff topology, two integral current spaces are close in the intrinsic flat topology whenever they can be isometrically embedded in a common complete metric space in which they are close in the usual flat topology \cite[Definition 1.4]{Antoine23a} .

A key result is Wenger's compactness  theorem:
\begin{theo}\label{compactness Wenger} \cite{Wenger11}\cite[Theorem 4.19]{SW11} \label{key compact}
Given a sequence of boundaryless integral current spaces 
$$(X_m,d_m,S_m)$$ with uniformly bounded mass and diameter, there is a subsequence converging to an integral current space in the intrinsic flat topology. 
\end{theo}

\subsection{Limits of almost Riemannian isometries and  intrinsic flat limit spaces}

As usual, inside an $n$-dimensional Riemannian manifold, we will denote by $\Vol$ and $\Area$ the $n$-dimensional and $(n-1)$-dimensional Hausdorff measure. Sometimes, we also use  $\mathcal{H}^k$ to denote the $k$-dimensional Hausdorff measure. 
Given a metric on a space, the standard notion of induced intrinsic metric is defined in \cite[Chapter 2, Section 2.3]{BBI22}. 
If $h$ is a Riemannian metric on a manifold $N$, let $\dist_{h}$ be the metric on $M$ induced by $g$. Sometimes we will make the identification
$$(N,h) = (N,\dist_h).$$
We will use a few times the following simple fact: if $(M,g)$ is a compact Riemannian $n$-manifold with a piecewise smooth metric $g$ inducing $\dist_g$, 
 then for any metric $d$ whose induced intrinsic  metric is $\dist_g$, and any open subset $\Omega\subset M$, the  mass of $\llbracket 1_\Omega \rrbracket$ as an $n$-dimensional current in $(M,d)$ is at most $\Vol(\Omega,g)$.

\begin{lemme}\label{lemma:limit map} \cite{Sor18, GCS23}
Let $(E_1,d_1),(E_2,d_2)$ be two complete metric spaces. Let $S_i$ be a sequence of integral currents in $(E_1,d_1)$ and let 
$$\varphi_i:\spt S_i \to (E_2,d_2)$$ be a sequence of $\lambda$-Lipschitz maps for some $\lambda>0$ independent of $i$. Suppose that $S_i$ (resp. $(\varphi_i)_\sharp S_i$) converges in the flat topology to an integral current $S_\infty$ (resp. $T_\infty$) inside $(E_1,d_1)$ (resp. inside $(E_2,d_2)$), and that 
$(E_2,d_2)$ is compact.

Then there is a $\lambda$-Lipschitz map 
$$\varphi_\infty:(\spt S_\infty,d_\infty)\to (E_2,d_2)$$ such that:
\begin{enumerate}
\item after taking a subsequence if necessary, for any positive integer $m$ and any collection of $m$ points $\{x_{\infty,1},...,x_{\infty,m}\} \subset \spt S_\infty$, there is a sequence of collections of $m$ points $\{x_{i,1},...,x_{i,m}\} \subset N_i$ such that  for each $j\in \{1,...,m\}$,  as $i\to \infty$, $x_{i,j}$ converges to $x_{\infty,j}$, and $\varphi_i(x_{i,j})$ converges to $\varphi_\infty(x_{\infty,j})$,
\item $(\varphi_\infty)_\sharp S_\infty = T_\infty$ as currents inside $(E_2,d_2)$.
\end{enumerate}

\end{lemme}

\begin{proof}
(1) is \cite[Theorem 6.1]{Sor18}, and is proved using an Arzel\`{a}-Ascoli type argument.

(2) follows from a slight generalization of \cite[Lemma 7.3]{GCS23}. If 
$L^\infty(E_2)$ is the Banach space of bounded real functions on $E_2$ endowed with the $L^\infty$ norm, then it is well-known that $(E_2,d_2)$ embeds isometrically inside $L^\infty(E_2)$ by the Kuratowski embedding, and  $L^\infty(E_2)$ is an injective metric space in the following sense: given any other metric space $Y$, a subset $A \subset Y$, and a $\lambda$-Lipschitz map $\phi : A \to  L^\infty(E_2)$, there exists an extension of $\phi$, called $\tilde{\phi} : Y \to L^\infty(E_2)$, which is still $\lambda$-Lipschitz. We can adapt the proof of \cite[Lemma 7.3]{GCS23} by using that extension theorem, instead of McShane's extension theorem.
\end{proof}

For this subsection, we will make the following assumption.

\begin{assumption} \label{assump}
Let $(N,h)$ be a connected, closed, oriented Riemannian $n$-manifold.
Let $S_i$ be a sequence of integral currents in a complete metric space $(E,d)$, converging in the flat topology to an integral current $S_\infty$ inside $(E,d)$.
Suppose that 
\begin{enumerate}[label=(\alph*)]
\item each support $N_i:=\spt S_i$, endowed with the intrinsic metric induced by the metric $d$, is a compact, oriented Riemannian manifold $(N_i,h_i)$ with a  piecewise smooth metric $h_i$ (possibly with nonempty piecewise smooth boundary),
\item $\lim_{i\to \infty} \Area(\partial N_i,h_i) =0$, 
\item there is a sequence of maps 
$$\varphi_i: (N_i,d\vert_{N_i}) \to (N,\dist_h)$$ which are $C^1$ on the smooth part of $N_i$ and $\lambda$-Lipschitz for some $\lambda>0$ independent of $i$, such that $(\varphi_i)_\sharp(S_i)$ converges to $\llbracket 1_N\rrbracket$ in the flat topology inside $(N,h)$,
\item  there is a sequence of open subsets $R_i$ contained in the part of $N_i$ where $h_i$ is smooth, such that 
 $\lim_{i\to \infty} \Vol(N_i\setminus R_i,h_i)  = 0$ and $\lim_{i\to \infty} \Vol(R_i,h_i)  = \Vol(N,h)$,
\item moreover, $\varphi_i$ is almost a Riemannian isometry on $R_i$ in the sense that
 $$\lim_{i\to \infty}  \big{\|}\sum_{u,v=1}^n | h(d\varphi_i(e'_u),d\varphi_i(e'_v)) - \delta_{uv} | \big{\|}_{L^\infty(R_i)}=0,$$
where $\{e'_u\}_{u=1}^n$ denotes any choice of orthonormal bases for the tangent spaces of $(N_i,h_i)$.
\end{enumerate}
\end{assumption}

Some of the conditions above are unnecessarily restrictive, but they will be convenient for our applications. Note that Lemma \ref{lemma:limit map} applies under Assumption \ref{assump} and yields a limit map
$$\varphi_\infty: \spt S_\infty \to (N,\dist_h).$$

The following proposition, while elementary, is technically important for us. It is related to, but different from  Lipschitz-volume rigidity results like \cite[Proposition C.1]{BCG95}, \cite[Sections 3, 4, 5]{BBCG12}, \cite[Theorem 1.1]{DNP23}, \cite[Theorem 1.1]{GCS23} and \cite[Theorem 1.2]{Zus23}.

\begin{prop} \label{proposition:limit map}
Suppose that Assumption \ref{assump} above holds and let 
$$\varphi_\infty: \spt S_\infty \to (N,\dist_h)$$ be the limit map constructed in Lemma \ref{lemma:limit map}.
\begin{enumerate}
\item
Then $\varphi_\infty$ is a bi-Lipschitz bijection and its inverse $\varphi_\infty^{-1}:(N,\dist_h) \to \spt S_\infty $ is $1$-Lipschitz with respect to the induced intrinsic metrics. 

\item Suppose additionally that for any $\epsilon>0$, there is $r_\epsilon>0$ such that if $i$ is large enough, then for any $x,y\in N_i$ such that $d(x,y)<r_\epsilon$, we have
$$\dist_{h}(\varphi_i(x),\varphi_i(y)) \leq (1+\epsilon)d(x,y).$$ 
Then $\varphi_\infty$ is an isometry with respect to the induced intrinsic metrics. 
\end{enumerate}
\end{prop}

\begin{remarque}
The limit map $\varphi_\infty$ in (1) is not $1$-Lipschitz for the intrinsic metrics  in general, which means that the additional condition in (2) is needed. Indeed consider for instance the standard round metric $g_{\mathrm{Eucl}}$ on the Euclidean unit sphere $S^2$, and for each $i>0$, consider the conformal metric $g_i:=f^2.g_{\mathrm{Eucl}}$ where $f:S^2\to [\frac{1}{2},1]$ is $1$ outside the $\frac{1}{i}$-neighborhood of the equator and $\frac{1}{2}$ in the $\frac{1}{2i}$-neighborhood of the equator. Let $\dist_{g_{\mathrm{Eucl}}}$ and $\dist_{g_i}$ be the corresponding intrinsic metrics. Then Assumption \ref{assump} is satisfied with $\varphi_i$ being  the identity map $\id:(S^2,\dist_{g_i})\to (S^2,\dist_{g_{\mathrm{Eucl}}})$, etc. However, the intrinsic flat limit and Gromov-Hausdorff limit of $(S^2,g_i)$ are both determined  by the length structure $L$ on $S^2$ induced by $\dist_{g_{\mathrm{Eucl}}}$ for curves not touching the equator, and with an equator of length $\pi$ instead of $2\pi$. The limit $\varphi_\infty$ is still  the identity map $\id:(S^2,L)\to (S^2,\dist_{g_{\mathrm{Eucl}}})$ and it is not $1$-Lipschitz for the intrinsic metrics.
\end{remarque}

\begin{proof}
Property (2) follows directly from property (1) in the statement and Lemma \ref{lemma:limit map} (1).  Indeed, applying the additional assumption in (2) with arbitrarily small $\epsilon$, together with Lemma \ref{lemma:limit map} (1), we obtain that $\varphi_\infty$ does not increase distances for the intrinsic metrics, in other words it is $1$-Lipschitz for the intrinsic metrics. Since property (1) says that the inverse of $\varphi_\infty$ is also $1$-Lipschitz for the intrinsic metrics, it is an isometry.

It remains to prove property (1). Note that by Lemma \ref{lemma:limit map} (1), $\varphi_\infty$ is $\lambda$-Lipschitz. 
Let $L_d$ be the intrinsic metric on $\spt S_\infty$ induced by the restricted metric $d\vert_{S_\infty}$ (a priori $L_d$ is allowed to take $\infty$ as value). Note that by Assumption \ref{assump} (c) (d) (e), the area formula and the lower semincontinuity of mass under flat convergence, we can assume that $\varphi_i$ is injective on $R_i$ without loss of generality by reducing that domain a bit.

For $\eta>0$, set
$$O_\eta:= \text{$\eta$-neighborhood of $\spt S_\infty$ inside $(E,d)$}.$$  
Then for every $\eta>0$,
\begin{equation}\label{eta to 0}
\lim_{i\to \infty}\|S_i\| (E\setminus O_{\eta})  =0,
\end{equation}
Indeed, let us assume  on the contrary that for some $\eta>0$, $\liminf_{i\to \infty}\|S_i\| (E\setminus O_{\eta})>0$. Then, by Assumption \ref{assump} (a) (d) (e), we should have 
\begin{align*}
 \liminf_{i\to \infty} \mathbf{M}((\varphi_i)_\sharp (S_i\llcorner O_\eta)) &=  \liminf_{i\to \infty} \mathbf{M}(S_i\llcorner O_\eta) \\
 & < \lim_{i\to \infty}\mathbf{M}(S_i) \\
 & = 
 \lim_{i\to \infty} \mathbf{M}((\varphi_i)_\sharp S_i) = \Vol(N,h).
 \end{align*}
By a standard application of the slicing theorem, we can assume without loss of generality that  the restricted current $S_i\llcorner (E\setminus O_{\eta})$ is an integral current converging to $0$ in the flat topology as $i\to \infty$. Thus $(\varphi_i)_\sharp (S_i\llcorner O_\eta)$ still converges to $\llbracket 1_N\rrbracket$ in the flat topology. By Assumption \ref{assump} (c) and lower semicontinuity of the mass with respect to flat or weak convergence,
$$\liminf_{i\to \infty} \mathbf{M}((\varphi_i)_\sharp (S_i\llcorner O_\eta))  \geq \mathbf{M}(\llbracket 1_N\rrbracket) = \Vol(N,h).$$ This contradicts the previous inequality and so (\ref{eta to 0}) was true.

Given a Lipschitz curve $\omega$ in $(E,d)$, let $\length_{(E,d)}(\omega)$ denote its length with respect to the metric $d$. Next, it is convenient to show the following ``curve lifting'' property.

\textbf{Curve lifting property:} Let $\eta>0$. Let $x,y\in \spt S_\infty$ and let 
$$l:=\dist_{h}(\varphi_\infty(x),\varphi_\infty(y)).$$ Then there exists 
a compact connected Lipschitz curve $\omega$ contained in $O_{\eta}$, starting at $x$, ending at $y$, and moreover 
$$\length_{(E,d)}(\omega)\leq l +\eta. $$

\begin{proof}[Proof of the curve lifting property]

We fix an $\eta>0$.
Let $b_x$, $b_y$ be the metric balls in $(E,d)$, of radius $\eta'>0$ chosen later, centered at $x,y\in \spt S_\infty$. By Lemma \ref{lemma:limit map} (1) and the fact that the $\varphi_i$ are assumed to be $\lambda$-Lipschitz (Assumption \ref{assump} (c)), for all $i$ large and every $q\in b_x \cap \spt S_i$ (resp. $q\in b_y\cap \spt S_i$), we have 
$$\dist_{h}(\varphi_i(q), \varphi_\infty(x)) \leq 2\eta'\lambda \quad \text{(resp. $\dist_{h}(\varphi_i(q), \varphi_\infty(y)) \leq 2\eta'\lambda$).}$$
By lower semicontinuity of the mass, for each $i$ large, 
$$\|S_i\|(b_x) >2\kappa \quad \text{and} \quad  \|S_i\|(b_y) >2\kappa$$ for some $\kappa>0$ depending on $\eta',x,y$ but independent of $i$.  For $i$ large, since we are assuming that $\varphi_i$ is injective on $R_i$, by Assumption \ref{assump} (d) (e) and the area formula, we have the following volume estimates:
\begin{equation}\label{quite important}
\begin{split}
\mathcal{H}^n\big(\varphi_i( R_i \cap b_x  )\big) & \geq \kappa,\\ \quad \mathcal{H}^n\big(\varphi_i( R_i \cap b_y  )\big) &\geq \kappa.
\end{split}
\end{equation}

Let us choose $\eta'$ small so that $2\eta'\lambda <r_1$ where $r_1$ is defined in Lemma \ref{tedii} (stated at  the end of this subsection) and depends on our fixed $\eta>0$.
Applying Lemma \ref{tedii} to $p=\varphi_\infty(x)$, $q=\varphi_\infty(y)$, 
$A_i=\varphi_i( R_i \cap b_x  )$, $B_i=\varphi_i( R_i \cap b_y  )$, we find for each $i$ large enough, two points 
$$y_{1,i} \in R_i \cap b_x,\quad y_{2,i} \in R_i \cap b_y$$ 
 and a smooth curve 
 $$\sigma_i \subset  N$$ with 
 $\length_h(\sigma_i)\leq l+\frac{\eta}{3}$, joining $\varphi_i(y_{1,i})$ to $\varphi_i(y_{2,i})$
such that
the restricted preimage 
$$\varkappa_i:=(\varphi_i)^{-1}(\sigma_i)$$ is a 
compact curve in $N_i$ avoiding $\partial N_i$,
whose endpoints satisfies 
$$\varphi_i(\partial \varkappa_i) \subset \varphi_i(y_{1,i}) \cup \varphi_i(y_{2,i}).$$
Because  $y_{1,i}$ and $y_{2,i} $ belong to $R_i$ (on which $\varphi_i$ is assumed to be injective), in fact
\begin{equation}\label{boundary R}
\partial \varkappa_i = \{y_{1,i},y_{2,i}\}.
\end{equation} 
By Assumption \ref{assump} (c) (d) (e), the restriction of $S_i$ to the complement of $R_i$ has mass converging to $0$ as $i\to \infty$; similarly, by (\ref{eta to 0}), the restriction of $S_i$ to the complement of $O_{\eta/2}$ has mass converging to $0$.
Thus, the second part of Lemma \ref{tedii} ensures that we can find such $\varkappa_i $ satisfying additionally
\begin{equation}\label{limlim}
\begin{split}
\lim_{i\to \infty}\mathcal{H}^1(\varkappa_i \setminus R_i) = 0, \quad \lim_{i\to \infty}\mathcal{H}^1(j_i(\varkappa_i) \setminus O_{\eta/2}) = 0.
\end{split}
\end{equation}
Together with Assumption \ref{assump} (e) and the area formula, these properties imply:
\begin{equation} \label{h1m}
\mathcal{H}^1(\varkappa_i) \leq (1+\epsilon_i)(l+\eta/2) +\epsilon_i
\end{equation}
where $\lim_{i\to \infty}\epsilon_i =0$. 
By using (\ref{limlim}) and the fact that $y_{1,i}$ (resp. $y_{2,i}$) is in $b_x$ (resp. $b_y$),
we easily construct a new curve $\omega_i$ fully contained in $O_\eta$ joining $x$ to $y$, with length at most $l+\eta$ for $i$ large. This proves the curve lifting property.
\end{proof}

\vspace{1em}

The curve lifting property implies the following  useful properties.
Firstly, $\spt S_\infty$ is compact.
Suppose towards a contraction that $\spt S_\infty$ is not compact, then for some $r'\in(0,1)$, there is an infinite sequence of points $\{x_m\}_{m\geq 0} \subset \spt S_\infty$ such that those points are pairwise at distance at least $r'$ in $(E,d)$. By compactness of $N$, for any $\epsilon>0$ there are $m_1\neq m_2$ such that $$\dist_{h}(\varphi_{\infty}(x_{m_1}),\varphi_{\infty}(x_{m_2}))\leq \epsilon.$$ Then the curve lifting property implies that the distance between  $x_{m_1}$ and $x_{m_2}$ is at most $\epsilon$, a contradiction when $\epsilon< r'/2$.

Secondly  $\varphi_{\infty} : \spt S_\infty \to N$ is bijective. 
Indeed we verify that $\varphi_{\infty}$ is injective by a direct application of the curve lifting property. Surjectivity follows from Lemma \ref{lemma:limit map} (2) and the compactness of $\spt S_\infty$.

We are ready to prove property (1) of our proposition. Take two points $u,v\in N$ and let $\eta>0$. Let $x:= \varphi_{\infty}^{-1}(u)$, $y:=\varphi_{\infty}^{-1}(v)$.
By applying the curve lifting property repeatedly and making $\eta\to 0$, by compactness of $\spt S_\infty$ we get a limit Lipschitz curve in $\spt S_\infty$ joining $x$ to $y$, with length at most $\dist_{h}(u,v)$.
Thus the inverse $\varphi_{\infty}$ is indeed $1$-Lipschitz for the intrinsic metrics, and $\varphi_\infty$ is bi-Lipschitz, as wanted.

\end{proof}

Below is a lemma based on the coarea formula, which was applied in the proof of Proposition \ref{proposition:limit map}.

\begin{lemme} \label{tedii}
Let $\eta>0$ and consider $p,q\in (N,h)$. Under Assumption \ref{assump}, if $r_1>0$ is small enough then the following holds. 
For each $i$, let $A_i, B_i$ be regions in $N$  contained in the $r_1$-neighborhoods of $p$ and $q$  respectively,  and  such that $\mathcal{H}^{n}(A_i)\geq\kappa$ and  $\mathcal{H}^{n}(B_i)\geq\kappa$ for some $\kappa>0$ independent of $i$. 
Then, there is a smooth curve $\sigma_i\subset N$ with $\length_h(\sigma_i)\leq \dist_h(p,q)+\frac{\eta}{10}$, whose endpoints $p_i,q_i$ are in $A_i,B_i$ respectively, such that  
 the preimage 
$$\varkappa_i := (\varphi_i)^{-1}(\sigma_i)$$ 
is a compact smooth curve (with possibly several connected components) in $N_i$ avoiding $\partial N_i$, and such that all its endpoints are sent  by $\varphi_i$ to $\{p_i,q_i\}$.

Moreover, if $Q_i\subset (N_i,h_i)$ are open regions whose $h_i$-volumes converge to $0$ as $i\to \infty$, then 
$\sigma_i$ can additionally be chosen so that the $h_i$-length of $\varkappa_i \cap Q_i$ tends to $0$ as $i\to \infty$.
\end{lemme}

\begin{proof}
There is a length-minimizing (thus necessarily embedded) geodesic segment $\gamma_{pq}$ in $(N,h)$ joining $p$ to $q$. 
Given $\eta>0$, we can find a Lipschitz diffeomorphism  $\Phi_\eta$
independent of $i$, from a neighborhood $U_{\gamma_{pq}}$ of $\gamma_{pq}$ to $[0,3]\times [0,1]^{n-1}$:
$$\Phi_\eta :U_{\gamma_{pq}} \to [0,3]\times [0,1]^{n-1}$$
and satisfying the following properties. 
It sends $p$ (resp. $q$) to $(\frac{1}{2},\frac{1}{2},...,\frac{1}{2})\in [0,1]\times[0,1]^{n-1}$ (resp. $(\frac{5}{2},\frac{1}{2},...,\frac{1}{2}) \in [2,3]\times[0,1]^{n-1}$), and for any $x\in [0,1]^{n-1}$,  
\begin{equation}\label{hgt}
\length_h(\Phi_\eta^{-1}([0,3]\times\{x\})) \leq \length_h(\gamma_{pq})+\frac{\eta}{100}.\end{equation}

Let $r_0>0$ be a small constant to be fixed later.
Fix a radius $r_1>0$ so small that the $r_1$-balls in $(N,h)$ around $p$ and $q$ are sent by $\Phi_\eta$ in the $r_0$-neighborhood of $\Phi_\eta(p)$ and $\Phi_\eta(q)$ respectively. 
Due to Assumption \ref{assump} (b) (c), 
\begin{equation}\label{qaz}
\lim_{i\to \infty} \mathcal{H}^{n-1}(\Phi_\eta( \varphi_i(\partial \Sigma_i)\cap U_{\gamma_{pq}})) =0.
\end{equation}
If $A_i,B_i\subset N$ are as in the statement, then for some $\kappa_0>0$ independent of $i$, for all $i$:
$$\quad \mathcal{H}^n(\Phi_\eta(A_i)) \geq \kappa_0, \quad \mathcal{H}^n(\Phi_\eta(B_i)) \geq \kappa_0.$$
Let $\mathrm{proj}:\mathbb{R}^n \to \mathbb{R}^{n-1}$ be the projection on the last $n-1$ coordinates. By Fubini's theorem, for each $i$ we can find a vector $\overrightarrow{t_i}\in \{0\}\times [-1,1]^{n-1}$ with  $|t_i|<r_0$ such that if we set
\begin{align*}
\mathcal{W}(\Phi_\eta(A_i), \Phi_\eta(B_i),t_i):=\{x\in [0,1]^{n-1}; \quad &   \mathcal{H}^1(\mathrm{proj}^{-1}(x) \cap \Phi_\eta(A_i)) >0 \text{ and}\\
&   \mathcal{H}^1(\mathrm{proj}^{-1}(x) \cap (\Phi_\eta(B_i)+\overrightarrow{t_i})) >0 \} 
\end{align*}
then we have
$$\mathcal{H}^{n-1}(\mathcal{W}(\Phi_\eta(A_i), \Phi_\eta(B_i),t_i) )>\kappa_1 $$
for some $\kappa_1>0$ independent of $i$. 
Since $|t_i|<r_0$, there is for each $i$ a diffeomorphism $F_i$ of $[0,3]\times [0,1]^{n-1}$, whose biLipschitz constant is bounded by $1+Cr_0$ for some constant $C>0$ independent of $i$, and such that 
\begin{equation}\label{zxc}
\mathcal{H}^{n-1}(\mathcal{W}(F_i\circ \Phi_\eta(A_i), F_i\circ \Phi_\eta(B_i),0) )>\kappa_1.
\end{equation}
Consider now 
$$\varphi'_i:= F_i \circ \Phi_\eta \circ \varphi_i.$$
These maps are uniformly Lipschitz, independently of $i$.
Using (\ref{zxc}), (\ref{qaz}) and applying the coarea formula and Sard's theorem twice, 
first to the maps $\mathrm{proj} \circ \varphi'_i$, then to the map $(x_1,...x_n)\mapsto x_1$, we find for each $i$ some straight segment $\hat{\sigma}_i$ in $[0,3]\times [0,1]^{n-1}$ joining $a_i\in F_i \circ \Phi_\eta(A_i)$ to $b_i\in F_i \circ \Phi_\eta(B_i)$, such that $\mathrm{proj}(\hat{\sigma}_i)$ is a point in $[0,1]^{n-1}$, and:
\begin{itemize}
\item $\sigma_i:=(F_i \circ \Phi_\eta)^{-1}(\hat{\sigma}_i)$
is a smooth connected curve in $(N,h)$ with endpoints in $A_i, B_i$ respectively. By (\ref{hgt}) and a compactness argument, its length is at most $\dist_h(p,q)+\frac{\eta}{10}$ if $r_0$ is chosen small enough to make the biLipschitz constants of $F_i$ close enough to $1$.
\item
$\varkappa_i := (F_i \circ \Phi_\eta \circ \varphi_i)^{-1}(\hat{\sigma}_i) =  (\varphi_i)^{-1}(\sigma_i) \subset N_i$
is a compact smooth curve (with possibly several connected components) which avoids $\partial\Sigma_i$ for all $i$ large enough: $\varkappa_i \cap  \partial \Sigma_i=\varnothing.$
\end{itemize}
Given our $\eta>0$, we choose $r_0$ small enough so that the first bullet is satisfied. Then the first part of the lemma holds when $r_1$ is small enough.

For the second part of the lemma, we apply the coarea formula again to $\varphi'_i$, which ensures that as $i\to \infty$, for some choice of $\hat{\sigma}_i$, $\varkappa_i$ intersects $Q_i$ on a set of arbitrarily small $h_i$-length since the $h_i$-volumes of $Q_i$ converge to $0$.

\end{proof}

\section{The spherical Plateau problem for hyperbolic manifolds} \label{section:Plateau}

\subsection{The spherical Plateau problem} \label{subsection:platee}

Let us define the spherical Plateau problem for closed oriented hyperbolic manifolds, which is part of a more general framework \cite[Section 3]{Antoine23a}.
Let $M$ be a closed oriented hyperbolic manifold, whose fundamental group is denoted by $\Gamma$. Let $S^\infty$ be the unit sphere in $\ell^2(\Gamma)$. The $\ell^2$-norm induces a Hilbert Riemannian metric $\mathbf{g}_{\mathrm{Hil}}$ on $S^\infty$. The group $\Gamma$ acts isometrically on $S^\infty$ by the (left) regular representation $\lambda_\Gamma:\Gamma\to \End(\ell^2(\Gamma))$: for all $\gamma\in \Gamma$, $x\in\Gamma$, $f\in S^\infty$,
$$(\lambda_\Gamma(\gamma).f)(x) := f(\gamma^{-1}x).$$
Since $\Gamma$ is infinite and torsion-free, $\Gamma$ acts properly and freely on the infinite dimensional sphere $S^\infty$. The quotient space $S^\infty/\lambda_{\Gamma}(\Gamma)$ is topologically a classifying space for $\Gamma$. It is also a Hilbert manifold endowed with the induced Hilbert Riemannian metric $\mathbf{g}_{\mathrm{Hil}}$. The diameter of $(S^\infty/\lambda_{\Gamma}(\Gamma),\mathbf{g}_{\mathrm{Hil}})$ is bounded from above by $\pi$.

Given base points $p_0\in M$, $q_0\in S^\infty/\lambda_{\Gamma}(\Gamma)$, there is a smooth immersion $M\to S^\infty/\lambda_{\Gamma}(\Gamma)$ inducing the identity map from $\pi_1(M,p_0)$ to $\pi_1(S^\infty/\lambda_{\Gamma}(\Gamma),q_0)$, which is unique up to homotopies sending $p_0$ to $q_0$. Other choices of $p_0,q_0$ yield homotopic maps, so that determines a unique homotopy class of maps which we call ``admissible''.
Set
$$\mathscr{H}_M := \{\phi : M \to  S^\infty/\lambda_{\Gamma}(\Gamma);\quad \text{ $\phi$ is an admissible smooth immersion}\}.$$
Any map $\phi\in \mathscr{H}_M$ defines the pull-back Riemannian metric $\phi^*(\mathbf{g}_{\mathrm{Hil}})$ on $M$.

Besson-Courtois-Gallot introduced the spherical volume of $M$ in \cite{BCG91}. It can be equivalently be defined as follows.
\begin{definition} \label{premiere def}
The \emph{spherical volume} of $M$ is defined as
$$\spherevol(M) := \inf\{\Vol(M,\phi^*(\mathbf{g}_{\mathrm{Hil}})); \quad \phi \in\mathscr{H}_M\}.$$
\end{definition}

The spherical volume of closed oriented hyperbolic manifolds was computed by Besson-Courtois-Gallot.
See \cite[Theorem 4.1]{Antoine23a} for the proof, adapted to our setting.
\begin{theo} \cite{BCG95,BCG96} \label{theorem:bcg}
Let $(M,g_0)$ be a closed oriented  hyperbolic manifold. Then 
\begin{equation} \label{bcg}
\spherevol(M) = \Vol(M,\frac{(n-1)^2}{{4n}}g_0).
\end{equation}
\end{theo}

A sequence $\phi_i\in \mathscr{H}_M$ is said to be \emph{minimizing} if  
$$\lim_{i \to \infty} \Vol(M,\phi_i^*(\mathbf{g}_{\mathrm{Hil}})) = \spherevol(M).$$ 
Denote by $\llbracket 1_M \rrbracket$ the integral current in $(M,g_0)$ induced by $M$ and its orientation. For a Lipschitz map $\phi:M\to S^\infty/\lambda_{\Gamma}(\Gamma)$, recall that $\phi_\sharp(\llbracket 1_M \rrbracket)$ denotes the push-forward integral current in  $S^\infty/\lambda_{\Gamma}(\Gamma)$.
We can now define spherical Plateau solutions.
\begin{definition}
We call spherical Plateau solution for $M$ any $n$-dimensional integral current space $C_\infty$ which is the limit in the intrinsic flat topology of a sequence $C_i := (\phi_i)_\sharp \llbracket 1_M\rrbracket$ where $\phi_i\in \mathscr{H}_M$ is a minimizing sequence.

\end{definition}
For any sequence $\phi_i\in \mathscr{H}_M$ such that
$$\lim_{i \to \infty} \Vol(M,\phi_i^*(\mathbf{g}_{\mathrm{Hil}})) = \spherevol(M),$$
the mass and diameter of $(\phi_i)_\sharp \llbracket 1_M\rrbracket$ are uniformly bounded, so by Wenger's compactness (Theorem \ref{key compact}) there is a subsequence of $(\phi_i)_\sharp \llbracket 1_M\rrbracket$ converging in the intrinsic flat topology. The need for an abstract compactness result like Theorem \ref{key compact} is explained in \cite[Remark 3.3]{Antoine23a}.

\begin{remarque}
While for our present purpose, it is enough to consider the set $\mathscr{H}_M $ of admissible smooth immersions from $M$ to $S^\infty/\lambda_{\Gamma}(\Gamma)$, we believe that it is more natural to formulate the general  spherical Plateau problem in terms of 
 integral currents with compact support in $S^\infty/\lambda_{\Gamma}(\Gamma)$ representing a homology class   $h \in H_n(\Gamma;\mathbb{Z})$.  This is the point of view presented in  \cite[Section 3]{Antoine23a}.
In fact, by \cite{Brunnbauer08} and a standard polyhedral approximation result for integral currents in Hilbert manifolds \cite[Lemma 1.6]{Antoine23a}, it is possible to prove that these two setups lead to the same notions of spherical volume and spherical Plateau solutions, at least when the countable group $\Gamma$ is torsion-free.
\end{remarque}

\subsection{The barycenter map and the Jacobian bound} \label{appendix b}

The barycenter map played a crucial role in the work of Besson-Courtois-Gallot on the volume entropy inequality \cite{BCG95,BCG96} (see also \cite{BCG99, Sambusetti99, CF03, Souto08} for a small sample of other uses of the barycenter map).

For the reader's convenience, all the main properties of the barycenter map are proved in our setting in \cite[Section 2]{Antoine23a} and the main Jacobian bound is recalled below. We choose to express the barycenter map using the $\ell^2$-space on a group, instead of the $L^2$-space on a boundary as in \cite{BCG95}. The advantage is that only a minimal amount of knowledge is needed, and that it extends directly to other more general situations (3-manifolds, connected sums, Plateau Dehn fillings, see \cite[Sections 4, 5, 6]{Antoine23a}). 

Let $(M,g_0)$ be a closed oriented hyperbolic manifold. Let $(\tilde{M},g_0)$ be its universal cover, namely the hyperbolic $n$-space. 
Let $\Gamma:=\pi_1(M)$. The latter acts properly cocompactly and freely on $(\tilde{M}, g_0)$.
Let $S^\infty$ be the unit sphere in the Hilbert space $\ell^2(\Gamma)$, on which $\Gamma$ acts freely and properly by isometries via the regular representation, so that $S^\infty/\lambda_{\Gamma}(\Gamma)$ is a smooth Hilbert manifold endowed with the standard round metric (see Subsection \ref{subsection:platee}).

Set $$\varkappa(t):= \frac{1}{c}\log(\cosh(ct))$$ 
where $c$ is a positive constant. When we fix $c$ large enough, the following holds: for any $w\in \tilde{M}$, the composition
$$\rho_{w}(.):= \varkappa(\dist_{g_0}(w,.))$$
is smooth everywhere and satisfies 
\begin{equation} \label{hessian lower bd''}
Dd\rho_{w} \geq \Id - d\rho_{w}\otimes d\rho_{w}.
\end{equation}

\begin{definition} \label{s+}
Fix a basepoint $o\in \tilde{M}$. Let $\mathbb{S}^+$ be the set of functions in $S^\infty$ with finite support.
For $f\in \mathbb{S}^+$, consider the functional
\begin{equation} \label{bf}
\begin{split}
\mathcal{B}_f & : \tilde{M}\to [0,\infty]\\
\mathcal{B}_f(x) & := \sum_{\gamma\in \Gamma} |f(\gamma)|^2 \rho_{\gamma.o}(x).
\end{split}
\end{equation}
The \emph{barycenter map} is then defined as
$$\mathrm{Bar} : \mathbb{S}^+\to \tilde{M}$$
$$\mathrm{Bar}(f) := \text{ the unique point minimizing $\mathcal{B}_f$}.$$
\end{definition} 
The barycenter map is well-defined: the modified distance functions $\rho_{\gamma.o}$ are strictly convex,  
moreover $\mathcal{B}_f$ tends to infinity uniformly as $x\to \infty$, so that the point where $\mathcal{B}_f$ attains its minimum exists and is unique. 
The subset $\mathbb{S}^+\subset S^\infty$ is invariant by $\Gamma$, and $\mathrm{Bar}$ is $\Gamma$-equivariant. 
The quotient map 
$\mathbb{S}^+/\Gamma \to M$ is also denoted by $\mathrm{Bar}$.  
For more details, see \cite[Section 2]{Antoine23a}.

We will avoid discussing regularity issues for the barycenter map $\mathrm{Bar}: \mathbb{S}^+/\Gamma \to M$ by only considering its restriction to the supports of ``polyhedral chains'', which will be enough in all our applications. A $k$-dimensional polyhedral chain $P$ in $S^\infty/\lambda_{\Gamma}(\Gamma)$ is by definition a $k$-dimensional integral current $P$ such that there are smoothly embedded totally geodesic $k$-simplices $S_1,...,S_m\subset  S^\infty/\lambda_{\Gamma}(\Gamma)$ endowed with an orientation, and integers $a_j$ so that
$$P= \sum_{j=1}^m  a_j \llbracket 1_{S_j}\rrbracket$$
(see \cite[Subsection 1.7]{Antoine23a}).
Given a polyhedral chain $P$ in $\mathbb{S}^+/\Gamma$, we can check that the restriction
$$\mathrm{Bar}: \spt(P) \to M$$
is indeed continuous and smooth on each simplex. For $1\leq k\leq n$, given a smooth embedding with totally geodesic image $\varphi: \mathbb{R}^k \to \mathbb{S}^+ \subset S^\infty$, 
let $Q$ be the tangent $k$-plane at $p:=\varphi(y)$ for some $y\in \mathbb{R}^k$.
The map 
$$\mathrm{Bar}: \varphi(\mathbb{R}^k)\to \tilde{M}$$ is smooth around $p$, and its differential along $Q$  is denoted by $d\mathrm{Bar}\big|_Q: Q\to T_{\mathrm{Bar}(p)} \tilde{M}$. 
For more details on those claims, see  \cite[Susbection 2.2]{Antoine23a}.

The main result in this Subsection is the following (see \cite[Lemma 2.4]{Antoine23a} for a proof):

 \begin{lemme} \cite{BCG95} \label{astuce}
Suppose that $n\geq 3$. Let $f\in \mathbb{S}^+$
 and let $Q$ be the tangent $n$-plane at $f$ of a totally geodesic $n$-simplex in $ \mathbb{S}^+$  passing through $f$. Then 
 \begin{equation} \label{jacbar<}
 |\Jac \mathrm{Bar}\big|_Q| \leq \big(\frac{{4n}}{(n-1)^2}\big)^{n/2}.
 \end{equation}
 
Moreover for any $\eta>0$ small enough,
there exists $c_{\eta}>0$ with $\lim_{\eta \to 0} c_{\eta} = 0$, such that the following holds. 
If 
 $$  |\Jac \mathrm{Bar}\big|_Q| \geq \big(\frac{{4n}}{(n-1)^2}\big)^{n/2} -\eta,$$
then for any norm $1$ tangent vector $\Vec{u}\in Q$,
 \begin{equation} \label{dbar>}
 |d \mathrm{Bar}\big|_Q(\Vec{u})| \geq  \big(\frac{4n}{(n-1)^2}\big)^{1/2} - c_{\eta}
 \end{equation}
 and
 for any
connected continuous piecewise geodesic curve $\alpha \subset \mathbb{S}^+$ of length less than $\eta$ starting at $f$, 
 we have 
  \begin{equation} \label{length<}
 \length_{g_0}(\mathrm{Bar}(\alpha)) \leq   (\big(\frac{4n}{(n-1)^2}\big)^{1/2}+  c_{\eta}) \length(\alpha)
  \end{equation}
  where $\length(\alpha)$ is computed using the standard round metric on $S^\infty$.
 \end{lemme}

\subsection{Intrinsic uniqueness for hyperbolic manifolds}

From a geometric point of view, a natural question is the uniqueness of spherical Plateau solutions for closed  hyperbolic manifolds.
We do not know if uniqueness holds, however we will prove uniqueness up to ``intrinsic isomorphism''.


Consider an integral current space $C=(X,d,T)$ and an oriented, connected, closed Riemannian manifold $(N,g_N)$, which induces the integral current space $(N,\dist_{g_N},\llbracket 1_N\rrbracket)$.
The intrinsic metric 
on $X$ induced by $d$ is denoted by $L_{d}$. 
Note that the identity map  
$$\id: (X,L_d)\to (X,d)$$ is always $1$-Lipschitz (on each path connected component). 
\begin{definition} \label{definition:intrinsic isomorphism}
We say that $C=(X,d,T)$ is \emph{intrinsically isomorphic} to $(N,g_N)$ if there is an isometry 
$$\varphi: (N,\dist_{g_N}) \to (X,L_d)$$ 
such that 
$$(\id\circ \varphi)_\sharp \llbracket 1_N\rrbracket = T.$$
\end{definition}
For clarity, we emphasize that ``being intrinsically isomorphic'' is weaker than ``being at intrinsic flat distance $0$ from each other''.

 Our main result in this section shows that in dimensions at least $3$, the spherical Plateau solutions for closed hyperbolic manifolds are unique up to intrinsic isomorphism,
 see Definition \ref{definition:intrinsic isomorphism}.

\begin{theo} \label{uniqueness1}
Let $(M,g_0)$ be a closed oriented hyperbolic manifold of dimension at least $3$. Then any spherical Plateau solution for $M$ is intrinsically isomorphic to $(M,\frac{(n-1)^2}{{4n}}g_0)$. 
\end{theo}

\begin{proof}

Let $\phi_i \in \mathscr{H}_M$ be a minimizing sequence, namely
\begin{equation}\label{itsminimizing}
\lim_{i \to \infty} \Vol(M,\phi_i^*(\mathbf{g}_{\mathrm{Hil}})) = \spherevol(M) = \Vol(M,\frac{(n-1)^2}{4n}g_0),
\end{equation}
where the second equality follows from Theorem \ref{theorem:bcg}.
We suppose that the integral currents 
$$C_i :=(\phi_i)_\sharp \llbracket 1_M\rrbracket$$ converge in the intrinsic flat topology to a spherical Plateau solution $$C_\infty=(X_\infty,d_\infty,S_\infty).$$ 
Set $\Gamma:=\pi_1(M)$. By a perturbation argument, we can assume without loss of generality that for all $i$, for all $y\in M$, any lift of $\phi_i(y)\in S^\infty/\lambda_{\Gamma}(\Gamma)$ in $ S^\infty\subset \ell^2(\Gamma)$ has finite support. In particular, we can assume that 
$$\spt(C_i) \subset \mathbb{S}^+/\Gamma$$ where $\mathbb{S}^+$ is defined in Subsection \ref{appendix b}. By a further perturbation of $\phi_i$, we can even assume that $C_i$ is a polyhedral chain (a notion defined in Subsection \ref{appendix b}), in particular that $\spt(C_i)$ is a finite union of embedded totally geodesic $n$-simplices in $S^\infty/\lambda_{\Gamma}(\Gamma)$, see \cite[Lemma 1.6]{Antoine23a}.
 
From now on, we will use the notation
 $$g':= \frac{(n-1)^2}{4n} g_0.$$
In the sequel, Jacobians, lengths and distances will be computed with respect to the metric $g'$ on $M$. 
Fix $o\in \tilde{M}$ 
and let $$\mathrm{Bar}:\mathbb{S}^+/\Gamma\to M$$ be the barycenter map, see Section \ref{appendix b}. 
By $\Gamma$-equivariance, for any $i$, $\mathrm{Bar}: \spt(C_i) \to M$ is a Lipschitz homotopy equivalence, and
\begin{equation} \label{1_M}
{\mathrm{Bar}}_\sharp(C_i) = \llbracket 1_M \rrbracket.
\end{equation}


By lower semicontinuity of the mass under intrinsic flat convergence \cite{SW11}:
\begin{equation} \label{masssinfinity}
\mathbf{M}(C_\infty) \leq \spherevol(M) =\Vol(M,g')
\end{equation}
(the equality above is Theorem \ref{theorem:bcg}).

The $n$-dimensional Jacobian of ${\mathrm{Bar}}$ along the  tangent $n$-plane of $\spt(C_i)$ at any point $q$ in the interior of a ``face'' of $\spt(C_i)$ is well-defined and is bounded from above by $1$ with respect to the metric $g'$ on $M$, by the main Jacobian bound (\ref{jacbar<}) in Lemma \ref{astuce}. 
This implies by the area formula 
and (\ref{1_M}) that
$$\mathbf{M}(C_i) \geq \Vol(M,g') = \spherevol(M).$$

Since $C_i$ has mass converging to $\spherevol(M)$, by the area formula, the Jacobian of ${\mathrm{Bar}}$ has to be close to $1$ on a larger and larger part of $\spt(C_i)$ as $i\to \infty$, meaning that there are open subsets $\Omega_i \text{ in the smooth part of } \spt(C_i)$
such that at every point $q\in \Omega_i$, 
 there is a well-defined tangent $n$-plane of $\spt C_i$ at $q$, 
and
\begin{equation} \label{omegaici}
\begin{split}
\lim_{i\to \infty} \mathbf{M}(C_i\llcorner \Omega_i)  = \lim_{i\to \infty} \mathbf{M}(C_i) & = \spherevol(M), \\
\lim_{i\to\infty} \| \Jac {{\mathrm{Bar}}} - 1\|_{L^\infty(\Omega_i)} & =0,\\
\end{split}
\end{equation}
where we recall that the Jacobian is computed with $g'$ and $\Jac {{\mathrm{Bar}}}$ denotes the Jacobian along the tangent $n$-plane, see Section \ref{appendix a}. For $r>0$, set 
$$\Omega_{i,r} := \text{the $r$-neighborhood of $\Omega_i$ in $\mathbb{S}^+/\Gamma\subset S^\infty/\lambda_{\Gamma}(\Gamma)$}.$$

By  (\ref{omegaici}), the coarea formula and Sard's theorem, 
after smoothing out the distance function from $\Omega_i$ by a standard argument and still using the notation ``$\Omega_{i,r}$'' for the $r$-sublevel set of the smoothed out distance function, there are $r^{(i)}\in (0,1)$ such that for each $i$, 
$$D_{i}:= C_i\llcorner \Omega_{i,r^{(i)}}$$
is an integral current, and $\spt(D_{i})$ is a compact piecewise smooth submanifold of $S^\infty/\lambda_{\Gamma}(\Gamma)$ satisfying the following:  
\begin{itemize}
\item
the boundary of $\spt(D_{i})$ is piecewise smooth (this is where considering the smoothed out distance function is used) and we have 
\begin{equation} \label{partial dki}
\lim_{i\to \infty} \mathbf{M}(\partial D_{i}) =0.
\end{equation}
\item after taking a subsequence, $D_{i}$ still converges to 
$$C_\infty=(X_\infty,d_\infty,S_\infty)$$ in the intrinsic flat topology as $i\to \infty$, In particular, there are a Banach space $\mathbf{Z}'$ and isometric embeddings 
$$\spt(D_{i}) \hookrightarrow \mathbf{Z}', \quad \spt S_\infty \hookrightarrow \mathbf{Z}'$$
(with a slight abuse of notations we consider those sets as subsets of $\mathbf{Z}'$), 
such that $D_i$ converges to $S_\infty$ in the flat topology inside $\mathbf{Z}'$.
\end{itemize}


By (\ref{omegaici}) and (\ref{1_M}), 
\begin{equation}\label{1300}
{\mathrm{Bar}}_\sharp(D_{i}) \quad \text{converges in the flat topology to $\llbracket 1_{M} \rrbracket.$ inside $(M,g')$} 
\end{equation}

Inequality (\ref{length<}) of Lemma \ref{astuce} ensures that a Lipschitz bound holds uniformly in a neighborhood of $\Omega_i$: for  any $\epsilon>0$, there is $r_\epsilon>0$, such that if $i$ is large enough, then for $f\in \Omega_i$ and $f' \in \mathbb{S}^+/\Gamma$ joined to $f$ by a piecewise geodesic curve $\alpha \subset  \mathbb{S}^+/\Gamma$ of length at most $r_\epsilon>0$, we have
\begin{equation}\label{dbar2}
 \length_{g'}({\mathrm{Bar}}(\alpha))   \leq  (1+\epsilon) \length(\alpha).
\end{equation}
Given $f,f'\in \mathbb{S}^+/\Gamma$ and a curve in $S^\infty/\lambda_{\Gamma}(\Gamma)$ joining those two elements, after a small perturbation, that curve can be assumed to be inside $\mathbb{S}^+/\Gamma$. 
As a consequence of (\ref{dbar2}), we get the following local Lipschitz bound: for any $\tilde{r}\in (0,1)$, the restriction of ${\mathrm{Bar}}$ to the subset $\Omega_{i,\tilde{r}}$ is  $\lambda$-Lipschitz for some $\lambda>0$ independent of $i$. 
In particular, the restriction 
\begin{equation} \label{dki}
{\mathrm{Bar}}: \spt(D_{i}) \to M \quad \text{is $\lambda$-Lipschitz.}
\end{equation}

We can now check that Assumption \ref{assump} is verified with $(N,h)=(M,g')$, $(E,d) = \mathbf{Z}'$, $S_i = D_i$, $N_i=\spt D_i$, $\varphi_i = \mathrm{Bar}$, $R_i = \Omega_i$. In particular, in order to check Assumption \ref{assump} (e), observe that since the Jacobian of $\mathrm{Bar}$ converges to $1$ on $\Omega_i$ by  (\ref{omegaici}),
$\mathrm{Bar}$ is forced to be almost a Riemannian isometry on $\Omega_i$ by (\ref{dbar>}), (\ref{length<}) in Lemma \ref{astuce}. Furthermore, the additional assumption in Proposition \ref{proposition:limit map} (2) is also satisfied by (\ref{dbar2}).

By Proposition \ref{proposition:limit map} (2), we immediately conclude that there is a limit map
${\mathrm{Bar}}_{\infty} : \spt S_\infty\to (M,g')$
which is an isometry for the intrinsic metrics induced on $\spt S_\infty$ and $M$. Moreover by Lemma \ref{lemma:limit map} (2), ${\mathrm{Bar}}_{\infty}$ preserves the current structures in the sense that 
$$
({\mathrm{Bar}}_{\infty})_\sharp(S_\infty) = \llbracket 1_{M} \rrbracket.
$$
In other words, $C_\infty= (X_\infty,d_\infty,S_\infty)$ is intrinsically isomorphic to $(M,g')$, as wanted.

\end{proof}

\section{The entropy stability theorem} \label{section:stability}

\subsection{Technical preparation}

As before, $M$ is the closed, connected, oriented hyperbolic manifold, $\Gamma$ is its fundamental group and $S^\infty$ is the unit sphere inside $\ell^2(\Gamma)$, which is acted upon by $\Gamma$ via the regular representation.

Let us define maps $\mathcal{P}_c$ relating the volume entropy of a Riemannian metric on $M$ and the spherical volume of $M$,  introduced by Besson-Courtois-Gallot, see \cite[Proof of Lemma 3.1]{BCG91}. 
Let $g$ be a Riemannian metric on $M$. The universal cover of $M$ is $\tilde{M}$ and its fundamental group is $\Gamma$. Let $h(g)$ be its volume entropy. Denote by $D_M$ a Borel fundamental domain in $\tilde{M}$ for the action of $\Gamma$ and let $\gamma.D_M$ be its image by an element $\gamma\in \Gamma$.  Besson-Courtois-Gallot considered  for $c>h(g)$ maps similar to the following:
\begin{align*}
\mathcal{P}_c &:  \tilde{M} \to  S^\infty \\
 & x\mapsto \{\gamma \mapsto \frac{1}{\|e^{-\frac{c}{2}\dist_{g}(x,.)}\|_{L^2(\tilde{M},g)}}   \big[\int_{\gamma. D_M}  e^{-c\dist_{g}(x,u)} \dvol_g(u)\big]^{1/2}\}.
\end{align*}
Those maps satisfy the following properties, which hold in any dimension $n\geq 2$: 
\begin{lemme} [\cite{BCG91}] \label{psic}
For a Riemannian metric $g$ on $M$, $\mathcal{P}_c$ is a $\Gamma$-equivariant Lipschitz map, and for almost any $x\in \tilde{M}$, it satisfies
 \begin{equation} \label{bhy}
 \sum_{j=1}^n |d_x{\mathcal{P}}_c(e_j)|^2 \leq \frac{c^2}{4},
 \end{equation}
 where $\{e_j\}$ is a $g$-orthonormal basis of $T_x\tilde{M}$.
\end{lemme}
\begin{proof}
For the reader's convenience, let us outline the proof. Consider $S_2(\tilde{M},g)$ the unit sphere in $L^2(\tilde{M},g)$. Set for $c>h(g)$:
$$\overline{\mathcal{P}}_c : \tilde{M} \to S_2(\tilde{M},g)$$
$$\overline{\mathcal{P}}_c : x\mapsto \{y\mapsto \frac{1}{\|e^{-\frac{c}{2}\dist_{g}(x,.)}\|_{L^2(\tilde{M},g)}}  e^{-\frac{c}{2}\dist_{g}(x,y)} \},$$
and set
$$\mathcal{I} : S_2(\tilde{M},g) \to S^\infty$$
 $$\mathcal{I} : f \mapsto \{\gamma \mapsto \big[\int_{\gamma.D_M} f^2(u) \dvol_g(u)  \big]^{1/2} \}.$$
 These maps are manifestly $\Gamma$-equivariant, and note that $\mathcal{P}_c = \mathcal{I} \circ \overline{\mathcal{P}}_c$.
 One easily checks that $\mathcal{I}$ is $1$-Lipschitz. To prove the lemma, it remains to study $\overline{\mathcal{P}}_c$. By the Pythagorean theorem, 
 \begin{align*}
 \|d_x\overline{\mathcal{P}}_c \|^2_{L^2(\tilde{M},g)} & \leq \frac{1}{\|e^{-\frac{c}{2}\dist_{g}(x,.)}\|^2_{L^2(\tilde{M},g)}} \int_{\tilde{M}} \|d_x e^{-\frac{c}{2}\dist_{g}(x,y)}\|^2 \dvol_g(y)\\
 & \leq  \frac{c^2/4}{\|e^{-\frac{c}{2}\dist_{g}(x,.)}\|^2_{L^2(\tilde{M},g)}} \int_{\tilde{M}} \| d_x \dist_{g}(.,y)\|^2 e^{-c\dist_{g}(x,y)}\dvol_g(y).
 \end{align*}
 Taking the trace and using that the norm of the gradient of the distance function is well-defined almost everywhere and equal to $1$, we get at almost every $x\in M$, in a  $g$-orthonormal basis $\{e_j\}$ of $T_x {M}$:
 $$\sum_{j=1}^n \|d_x\overline{\mathcal{P}}_c(e_j)\|^2_{L^2(\tilde{M},g)} \leq \frac{c^2}{4}.$$
 This proves the lemma.
 
\end{proof}

 If $g$ is a Riemannian metric on $M$, let $\dist_{g}$ be the geodesic distance on $M$ induced by $g$.
 The  definition of the standard notions of $\epsilon$-isometry, $\epsilon$-net can be found in \cite[Definition 7.3.27, Definition 1.6.1]{BBI22}. 
 Given  $\Omega$ subset of a Riemannian manifold $(M,g)$, $g\vert_\Omega$ denotes (by a slight abuse of notation) the intrinsic metric induced by the Riemannian metric $g$ using paths inside $\Omega$. In general $(\Omega,g\vert_\Omega)$ is very different from $(\Omega,\dist_{g}\vert_\Omega)$, where $\dist_{g}\vert_\Omega$ is the restriction of the  induced metric $\dist_g$ of $(M,g)$ to $\Omega$.

The set of admissible maps $\mathscr{H}_M$ was defined in Subsection \ref{subsection:platee}. The barycenter map $\mathrm{Bar} : \mathbb{S}^+/\Gamma\to M$ was defined in Subsection \ref{appendix b}. The following result is an intermediate step towards Theorem \ref{theorem:stable}, and its proof is parallel to that of Theorem \ref{uniqueness1} but more technical. 
 \begin{theo} \label{technical uniqueness1}
Let $(M,g_0)$ be a closed oriented hyperbolic manifold of dimension $n\geq 3$.  Let $g_i$ $(i\geq 1)$ be a sequence of Riemannian metrics on $M$ of same volume as $g_0$, and suppose that 
$$\lim_{i\to \infty} h(g_i) = h(g_0)=n-1.$$
Then, there are smooth open subsets $A_i\subset M$ such that the following holds after taking a subsequence:
\begin{enumerate}
\item 
$\lim_{i\to \infty}\Vol(A_i,g_i) = \Vol(M,g_0) $ and $\lim_{i\to \infty}\Area(\partial A_i,g_i) = 0$,
\item
$(A_i,g_i\vert_{A_i})$ converges in the intrinsic flat topology to an integral current space $$C_\infty=(X_\infty,d_\infty,S_\infty),$$
\item there is a bi-Lipschitz bijection 
$$\Psi: (M,g_0) \to (\spt S_\infty,d_\infty)$$ which is $1$-Lipschitz, and  $$\Psi_\sharp(\llbracket 1_M\rrbracket) = S_\infty.$$
\item  $(A_i,g_i\vert_{A_i})$ converges to $(\spt S_\infty,d_\infty)$ in the Gromov-Hausdorff topology. Moreover, for any $\varepsilon>0$, for all $i$ large enough, there is a homotopy equivalence 
$$f_i: M\to \spt S_\infty$$
such that 
the restriction $f_i : (A_i, {g_i}\vert_{A_i}) \to (\spt S_\infty,d_\infty)$ is an $\varepsilon$-isometry. 
\end{enumerate}

\end{theo}

\begin{proof}

\textbf{Step 1: Finding good subsets}

For technical convenience, set
$$g' := \frac{(n-1)^2}{4n}g_0,\quad g'_i:=\frac{(n-1)^2}{4n}g_i.$$
Note that after rescaling, 
$$h(g') = 2\sqrt{n}.$$
By our assumptions, there is a sequence $c'_i>h(g'_i)$ of positive numbers such that
\begin{equation}\label{cv n-1}
\lim_{i\to \infty} c'_i = 2\sqrt{n}.
\end{equation}
By Lemma \ref{psic}, the maps 
$$\mathcal{P}_{c'_i} : \tilde{M} \to S^\infty$$
are $\Gamma$-equivariant.
The quotient maps
$$ \mathcal{P}_{c'_i} : (M,g'_i)\to S^\infty/\lambda_{\Gamma}(\Gamma)$$
can be perturbed to be smooth  immersions. Those new maps now belong to $\mathscr{H}_M$. After a further small perturbation, we obtain homotopic smooth immersions 
$$\phi_i\in \mathscr{H}_M$$ sending $(M,g'_i)$ inside $\mathbb{S}^+/\Gamma$, see \cite[Lemma 1.6]{Antoine23a}. Moreover, by (\ref{bhy}) and (\ref{cv n-1}),  it is not hard to ensure that after those standard smoothings, for all $x\in M$:
\begin{equation}\label{nui}
\sum_{j=1}^n |d_x\phi_i(e'_j)|^2 \leq n+\nu_i
\end{equation}
for some positive $\nu_i \to 0$ (with respect to $g'_i$), where $\{e'_j\}$ is an orthonormal basis for $g'_i$. By (\ref{nui}) and the inequality of arithmetic and geometric means, 
\begin{equation}\label{metlq}
|\Jac \phi_i| \leq (1+\frac{\nu_i}{n})^{n/2}
\end{equation} 
where the Jacobian is computed with respect to $g'_i$.
By Theorem \ref{theorem:bcg}, 
$$\spherevol(M) = \Vol(M,g'), $$ on the other hand we have $\Vol(M,g'_i) = \Vol(M,g')$ by assumption. Hence, by (\ref{metlq}), $|\Jac \phi_{i} |_{g'_i}$ converges to $1$ on an open region $\hat{\Omega}_i\subset M$ with 
$$\lim_{i\to \infty}\Vol(\hat{\Omega}_i,g'_i) =\Vol(M,g'),$$
 which by (\ref{nui}) forces  
\begin{equation}\label{gradient}
\lim_{i\to \infty}  \big{\|}\sum_{u,v=1}^n| \mathbf{g}_{\mathrm{Hil}}(d\phi_i(e'_u),d\phi_i(e'_v)) - \delta_{uv} | \big{\|}_{L^\infty(\hat{\Omega}_i)}=0
\end{equation}
where $\mathbf{g}_{\mathrm{Hil}}$ is the standard Hilbert Riemannian metric on the spherical quotient $S^\infty/\lambda_{\Gamma}(\Gamma)$, and $\{e'_u\}_{u=1}^n$ denotes any choice of orthonormal bases for the tangent spaces of $(M,g'_i)$.


Exactly as in the proof of Theorem \ref{uniqueness1} and using (\ref{gradient}), we first find open subsets $\Omega_i \subset M$ with 
\begin{equation}\label{volomage}
\lim_{i\to \infty}\Vol(\Omega_i,g'_i) =\Vol(M,g'),
\end{equation}
which satisfy
$$\lim_{i\to \infty}  \big{\|}\sum_{u,v=1}^n| g'(d (\mathrm{Bar}\circ \phi_i)(e'_u),d (\mathrm{Bar} \circ \phi_i)(e'_v)) - \delta_{uv} | \big{\|}_{L^\infty(\Omega_i)}=0.$$
Then we  define smoothings of $r^{(i)}$-neighborhoods of $\Omega_i$ in $(M,g'_i)$, called $\Omega_{i,r^{(i)}}$, so that the closure of $\Omega_{i,r^{(i)}}$ is a compact manifold with a smooth boundary whose area $\Area(\partial \Omega_{i,r^{(i)}}, g'_i)$ goes to $0$ as $i\to \infty$, and the restriction $\mathrm{Bar} \circ \phi_i \vert_{\Omega_{i,r^{(i)}}}$ is uniformly Lipschitz.

\textbf{Step 2: Constructing the limit map}

We set
$$(N_i,h_i) := (\Omega_{i,r^{(i)}},g'_i\vert_{\Omega_{i,r^{(i)}}}).$$
In order to apply Wenger's compactness theorem, we need a uniform diameter bound. For that reason, if $\dist_{h_i}$ denotes the intrinsic metric induced by $g'_i$ using paths contained in $N_i$, we set
$$\hat{d}_i :=  \min\{\dist_{h_i}, 6\diam(M,g')\}.$$ This defines a metric on $N_i$ with diameter at most $6\diam(M,g')$, and it is locally isometric to the induced intrinsic metric  $h_i$. 
We then set 
$$D_{i} := \llbracket 1_{N_i}\rrbracket.$$

By Wenger's compactness theorem, the integral current spaces $D_i$  converge to an integral current space $$\hat{C}_\infty=(\hat{X}_\infty,\hat{d}_\infty,\hat{S}_\infty)$$ in the intrinsic flat topology, after picking a subsequence if necessary. In particular, there are a Banach space $\hat{\mathbf{Z}}$, and isometric embeddings
\begin{equation}\label{could have}
(N_i,\hat{d}_i) \hookrightarrow \hat{\mathbf{Z}},\quad \spt \hat{S}_\infty \hookrightarrow \hat{\mathbf{Z}},
\end{equation}
with the usual abuse of notations,
such that $\llbracket 1_{N_i}\rrbracket$ converges to $\hat{S}_\infty$ in the flat topology inside $\hat{\mathbf{Z}}$.

Next, we check that Assumption \ref{assump} is satisfied for 
$$(N,h)=(M,g'),\quad  S_i=D_i,\quad (N_i,h_i) = (\Omega_{i,r^{(i)}},g'_i\vert_{\Omega_{i,r^{(i)}}}),$$
$$  \varphi_i = \mathrm{Bar} \circ \phi_i ,\quad  R_i=\Omega_i...$$ (Note however that the additional condition of Proposition \ref{proposition:limit map} (2) is a priori not satisfied, which accounts for the difference between the statements of Theorem \ref{uniqueness1} and Theorem \ref{technical uniqueness1}.) Thus by Proposition \ref{proposition:limit map} (1), there is a limit map 
 $$\varphi_\infty : (\spt \hat{S}_\infty,\hat{d}_\infty) \to (M,g') $$
 which is Lipschitz,  bijective and whose inverse 
 $$\hat{\Psi} := \varphi_\infty^{-1}$$
  is $1$-Lipschitz with respect to the intrinsic metrics. Hence, $\hat{\Psi} $ is clearly $1$-Lipschitz and bi-Lipschitz.
By Lemma \ref{lemma:limit map} (2), $\hat{\Psi} _\sharp (\llbracket 1_M\rrbracket) =\hat{S}_\infty$.

 \textbf{Step 3: Convergence for the original induced metric}
 
We also need to check that $(N_i, \dist_{h_i},\llbracket 1_{N_i}\rrbracket)$, not just $(N_i,\hat{d}_i,\llbracket 1_{N_i}\rrbracket)$, subsequentially converges to the integral current space $\hat{C}_\infty$. 
Notice that for any $x\in N_i$ and $R\in(0,3\diam(M,g'))$, the metric balls $B_{\dist_{h_i}}(x,R) \subset (N_i,\dist_{h'_i})$ and $B_{\hat{d}_i}(x,R) \subset (N_i,{\hat{d}_i})$ are globally isometric. In particular, since $(\spt \hat{S}_\infty,d_\infty)$ has diameter at most that of $(M,g')$ by $1$-Lipschitzness of $\Psi$, if $O_r$ denotes the $r$-neighborhood of $\spt \hat{S}_\infty$ in $\hat{\mathbf{Z}}$, we have:
whenever $r\in (0,\diam(M,g'))$, for any $i$ and pair of points $x,y\in N_i \cap O_r$,  
 \begin{equation} \label{kpo}
 \dist_{h_i}(x,y) = \hat{d}_i(x,y).
 \end{equation}
By the slicing theorem for metric currents,
we can choose for each $i$,  some radius $r_i\in (0,\diam(M,g'))$ converging to $0$, such that if we set
$$\tilde{O}_i := O_{r_i} \cap N_i\subset \hat{\mathbf{Z}},$$
then $\llbracket 1_{\tilde{O}_i}\rrbracket$ are integral currents in $\hat{\mathbf{Z}}$ converging to $ \hat{S}_\infty$ in the flat topology.  By (\ref{kpo}), this means that $(\tilde{O}_i,  \dist_{h_i}\vert_{\tilde{O}_i},\llbracket 1_{\tilde{O}_i}\rrbracket )$ converges  to $\hat{C}_\infty$ in the intrinsic flat topology. 
We deduce in particular that the push-forward of $\llbracket 1_{\tilde{O}_i}\rrbracket $  by $\mathrm{Bar} \circ \phi_i$  converges to $\llbracket 1_M \rrbracket$ as currents in $(M,g')$. Then the liminf as $i\to \infty$ of the mass of this push-forward is at least  $\Vol(M,g')$ by lower semicontinuity of the mass. 
By the Jacobian bounds (\ref{metlq}), (\ref{jacbar<}), and since by (\ref{volomage}) we have $\lim_{i\to \infty} \Vol(N_i,h_i) = \Vol(M,g'),$ 
\begin{equation} \label{hopee}
\lim_{i\to \infty}\Vol(\tilde{O}_i ,h_i) = \Vol(M,g'),\quad \lim_{i\to \infty}\Vol(N_i\setminus \tilde{O}_i ,h_i) =0.
\end{equation}
We conclude that $(N_i, \dist_{h_i},\llbracket 1_{N_i}\rrbracket)$ converges to the same limit as $(\tilde{O}_i,  \dist_{h_i}\vert_{\tilde{O}_i},\llbracket 1_{\tilde{O}_i}\rrbracket )$ in the intrinsic flat topology, which is $\hat{C}_\infty$, as desired.

\textbf{Step 4: Gromov-Hausdorff convergence and $\varepsilon$-isometries}

In general, $(N_i, \dist_{h_i})$ does not converge in the Gromov-Hausdorff topology to $(\spt \hat{S}_\infty,d_\infty)$. The end of the proof is about fixing this issue. By (\ref{could have}), there are  finite subsets $\Sigma_i\subset N_i$ converging in the Hausdorff topology to $\spt \hat{S}_\infty$ in $\hat{\mathbf{Z}}$. For any $t>0$, let 
$$\Sigma_{i,t} := \text{$t$-neighborhood of $\Sigma_i$ in $(N_i,\dist_{h_i})$}.$$ By lower semicontinuity of the mass and (\ref{could have}),  for any $s_1>2t>0$ and any sequence of points $p_i\in \Sigma_{i,t}$, 
  \begin{equation}\label{khov0}
 \liminf_{i\to \infty} \Vol(B_{\dist_{h_i}\vert_{\Sigma_{i,t}}}(p_i,s_1),h_i) >\kappa_0(s_1)>0
  \end{equation}
  for some $\kappa_0(s_1)$ not depending on $t$. 
We also have the following stronger property:  
 for any $s_1>2t>0$ and any sequence of points $p_i\in \Sigma_{i,t}$, 
 \begin{equation}\label{khov}
 \liminf_{i\to \infty} \Vol(B_{h_i\vert_{\Sigma_{i,t}}}(p_i,s_1),h_i) >\kappa(s_1)>0
 \end{equation}
 for some $\kappa(s_1)$ not depending on $t$. Note that this is indeed a stronger inequality, since ${h_i\vert_{\Sigma_{i,t}}}$ is the intrinsic metric on $\Sigma_{i,t}$ induced by $h_i$ using paths inside $\Sigma_{i,t}$, and 
 $$B_{h_i\vert_{\Sigma_{i,t}}}(p_i,s_1) \subset B_{\dist_{h_i}\vert_{\Sigma_{i,t}}}(p_i,s_1).$$  
To check this stronger property, recall that $\spt \hat{S}_\infty$ has been shown to be bi-Lipschitz to the closed Riemannian manifold $(M,g')$ via a map $\varphi_\infty$.
 For any two points $a,b\in \Sigma_{i,t}$, we can find $a',b'\in\spt \hat{S}_\infty\subset \hat{\mathbf{Z}}$ approximating $a,b$. Then, given a minimizing geodesic segment $\gamma$ in $(M,g')$ between $\varphi_\infty(a'),\varphi_\infty(b')$, we can approximate $(\varphi_\infty)^{-1}(\gamma)$ by a curve in $\Sigma_{i,t}$ between $a,b$ without increasing the length by more than a constant factor. Hence, for $i$ large, 
 $$B_{\dist_{h_i}\vert_{\Sigma_{i,t}}}(p_i,\lambda_0 s_1) \subset B_{h_i\vert_{\Sigma_{i,t}}}(p_i,s_1)$$
 for some $\lambda_0\in (0,1)$ independent of $t$. This and (\ref{khov0}) explain (\ref{khov}).

 By (\ref{hopee}) and by the coarea formula, there is a $t>0$, arbitrarily small, such that
 $$\lim_{i\to \infty} \Vol(\Sigma_{i,t},h_i) = \Vol(M,g'),\quad \lim_{i\to \infty} \Area(\partial \Sigma_{i,t},h_i)=0.$$
This means that after taking a subsequence, we find $t_i>0$ converging to $0$ so that  if we set
$$A_i:=\Sigma_{i,t_i}$$
then  for any $s_1>0$ and any sequence of points $p_i\in A_i$, 
  \begin{equation}\label{hopee1}
  \liminf_{i\to \infty} \Vol(B_{h_i\vert_{A_i}}(p_i,s_1), h_i)  >0,
   \end{equation}
$$  \lim_{i\to \infty} \Vol(A_i,h_i)  = \Vol(M,g'),$$
 $$ \lim_{i\to \infty} \Area(\partial A_i,h_i)=0.$$
 Now we can reapply all the arguments in \textbf{Step 2} and \textbf{Step 3} to a smoothing of $A_i = \Sigma_{i,{t_i}}$ instead of $N_i$.
Let us summarize what we have achieved so far: 
subsequentially,  $(A_i, g'_i\vert_{A_i})$ converges to an integral current space 
$$C_\infty=(X_\infty,d_\infty,S_\infty),$$
 and there are a Banach space $\mathbf{Z}$, and isometric embeddings
\begin{equation}\label{could have bis}
(A_i, g'_i\vert_{A_i}) \hookrightarrow \mathbf{Z},\quad \spt S_\infty \hookrightarrow \mathbf{Z},
\end{equation}
with the usual abuse of notations,
such that $\llbracket 1_{A_i}\rrbracket$ converges to $S_\infty$ in the flat topology inside $\mathbf{Z}$. Moreover, there is a bi-Lipschitz, $1$-Lipschitz map 
$$\Psi : (M,g') \to (\spt S_\infty,d_\infty)$$
which is the inverse of a limit map constructed using Lemma \ref{lemma:limit map} applied to $\mathrm{Bar}\circ \phi_i$. The following analogue of (\ref{hopee}) holds: for any $r>0$, if $O_r$ is the $r$-neighborhood of $\spt S_\infty$ in $\mathbf{Z}$,
\begin{equation}\label{hopee3}
\lim_{i\to \infty} \Vol(A_i \setminus O_r, g'_i\vert_{A_i}) =0.
\end{equation}
The key additional  property we gained is that $(A_i, g'_i\vert_{A_i})$ now converges to $\spt S_\infty$ in the Hausdorff topology inside $\mathbf{Z}$,  by (\ref{hopee1}) and (\ref{hopee3}).  Note that in general,  $\spt S_\infty$ and the previous space $\spt \hat{S}_\infty$ could be very different. 

We can then set
$$f_i:= \Psi\circ \mathrm{Bar} \circ \phi_i: M\to \spt S_\infty,$$
which is a homotopy equivalence. By Lemma \ref{lemma:limit map} (1), we conclude that for any $\varepsilon>0$, 
$$f_i:(A_i, g'_i\vert_{A_i}) \to  (\spt S_\infty,d_\infty)$$ is an $\varepsilon$-isometry if $i$ is large. 
All of these complete the proof, after rescaling all the Riemannian metrics by $\frac{4n}{(n-1)^2}$.
\end{proof}

\subsection{Equidistribution of geodesic spheres in hyperbolic manifolds} \label{appendix c}

Consider $(M,g_0)$ a closed hyperbolic manifold, with universal cover $\tilde{M}$. Fix $\mathbf{x}\in M$ and let $\tilde{\mathbf{x}}$ be a lift of $\mathbf{x}$ by the natural projection $\tilde{M}\to M$. Let $T^1M$ denote the unit tangent bundle of $M$. 
Let $\tilde{S}(\tilde{\mathbf{x}},t)$ be the geodesic sphere of radius $t$ centered at $\tilde{\mathbf{x}}$ in $\tilde{M}$, and let $\tilde{S}_1(\tilde{\mathbf{x}},t)$ be its lift to the unit tangent bundle $T^1\tilde{M}$ by considering the outward unit normal vectors on $\tilde{S}(\tilde{\mathbf{x}},t)$. 
Let $S(\mathbf{x},t)$ denote the projection of $\tilde{S}(\tilde{x},t)$ in $M$, and let $S_1(\mathbf{x},t)$ be the projection of $\tilde{S}_1(\tilde{\mathbf{x}},t)$ to the unit tangent bundle $T^1M$. A measure on $T^1M$ (resp. ${S}_1({x}_0,t)$) is called invariant if it is induced by a measure on $T^1\tilde{M}$ invariant by isometries of $\tilde{M}$ (resp. induced by a measure on $\tilde{S}_1(\tilde{\mathbf{x}},t)$ invariant by rotations of center $\tilde{\mathbf{x}}$ in $\tilde{M}$).

As a corollary of the mixing property for the geodesic flow on closed hyperbolic manifolds, the lift of geodesic spheres equidistribute in the unit tangent bundle. This is for instance explained in \cite[Section 2]{EskinMcMullen93} for surfaces and generalized in \cite[Theorem 1.2]{EskinMcMullen93} \footnote{I thank Ben Lowe for pointing out this reference.}. With the above notations, the statement is the following:

\begin{theo}\label{EM}
For any continuous function $f:T^1M\to\mathbb{R}$,
$$\lim_{t\to \infty} \int_{{S}_1({x}_0,t)} f(y) d\mu_t(y) = \int_{T^1M} f(y) dv_{T^1M}(y),$$
where $d\mu_t$ is the unique invariant probability measure on ${S}_1({x}_0,t)$ and $dv_{T^1M}$ is the unique invariant probability measure on $T^1M$.

\end{theo}

Below, areas (namely $(n-1)$-dimensional Hausdorff measures) and lengths are computed using $g_0$. Given an open subset $U\subset M$, let $\pi_1(M,U)$ denote the relative homotopy group. Consider a (not necessarily length minimizing) geodesic segment $\sigma$ in $(M,g_0)$ with two different endpoints $x,y\in M$ and let $U_x,U_y$ be two disjoint open geodesic balls centered at $x$ and $y$. Fix $\tilde{\mathbf{x}}\in \tilde{M}$ as before.  Let
$$\pi:\tilde{M} \to M$$ 
be the natural projection.
\begin{coro} \label{coro equidistribution}
There is $\theta\in(0,1)$ depending on $M,\sigma, U_x,U_y$ such that for all $t$ large enough, there is an open subset $W_t\subset \tilde{S}(\tilde{\mathbf{x}},t)$ satisfying 
$$\Area(W_t) \geq \theta \Area(\tilde{S}(\tilde{\mathbf{x}},t))$$
and with the following property:
for any $z\in W_t$, if $l : [0,t] \to \tilde{M}$ denotes the length minimizing geodesic from $\tilde{\mathbf{x}}$ to $z$ in $(\tilde{M},g_0)$ parametrized by arclength, there are disjoint intervals
$$[a_1,b_1],...,[a_m,b_m] \subset [0,t]$$
such that 
\begin{enumerate}
\item $ \sum_{j=1}^m |b_j-a_j| \geq \theta t$,
\item for $j\in \{1,...,m\}$, the endpoints satisfy $\pi(l(a_j)) \in U_x$ and $\pi(l(b_j)) \in U_y$,
\item   for  $j\in \{1,...,m\}$, $\pi\circ l : [a_j,b_j]\to M$ is a geodesic segment joining $\pi(l(a_j))$ to $\pi(l(b_j))$, which is in the same class as $\sigma$ in $\pi_1(M,U_x\cup U_y)$.
\end{enumerate} 
\end{coro}

\begin{proof}
Let $t_0$ be the length of $\sigma$. By continuity, there exist an open subset $O$ of the unit tangent bundle $T^1M$ depending only on $M,U_x,U_y$, 
such that for any tangent vector $v$ in $O$, the basepoint $p$ of $v$ lies in $U_x$, and the geodesic $\gamma$ starting at $p$ with direction $v$ and length $t_0$ ends at a point $q\in U_y$, and satisfies the following:
$$\gamma \in [\sigma] \in \pi_1(M,U_x\cup U_y).$$
Informally, geodesics of length $t_0$ starting at a vector in $O$ stay ``close'' to $\sigma$.

Let $\tilde{\mu}_t$ be the invariant probability measure on $\tilde{S}_1(\tilde{\mathbf{x}},t)$ and set 
$$\tilde{O}:=\pi^{-1}(O) \subset T^1\tilde{M},\quad \tilde{O}_t:= \tilde{O}\cap \tilde{S}_1(\tilde{\mathbf{x}},t).$$
Recall that $\tilde{S}_1(\tilde{\mathbf{x}},t)$ is the lift of the sphere $\tilde{S}(\tilde{\mathbf{x}},t)$ by its normal unit vector. Below, by abuse of notations, we will identify  $\tilde{S}_1(\tilde{\mathbf{x}},t)$ and  $\tilde{S}(\tilde{\mathbf{x}},t)$.
By applying Theorem \ref{EM} to the characteristic function of $O$, for all $t$ large enough,  
\begin{equation}\label{c10}
\tilde{\mu}_{t}(\tilde{O}_t)>c_1>0
\end{equation}
for some $c_1$ independent of $t$. If $x\in \tilde{S}_1(\tilde{\mathbf{x}},t)$, let $\tau_x$ be the unique geodesic segment from the basepoint $\tilde{\mathbf{x}}$ to $x$ in $\tilde{M}$.

We claim that for some $c_2,c_3>0$, for any large integer $N$,
\begin{align*}
\tilde{\mu}_{Nt_0}(&\{x\in \tilde{S}_1(\tilde{\mathbf{x}},Nt_0): \text{ for at least $c_2N$ distinct $k\in \{1,...,N\}$,  $\tau_x \cap \tilde{O}_{kt_0}\neq\varnothing$}\}) > c_3.
\end{align*}
Roughly speaking, this inequality means that for a uniformly positive fraction of the sphere $\tilde{S}_1(\tilde{\mathbf{x}},Nt_0)$, geodesics  from $\tilde{\mathbf{x}}$ to that portion of $\tilde{S}_1(\tilde{\mathbf{x}},Nt_0)$ stay close to $\sigma$ on a uniformly positive fraction of their length. Before proving the claim, note that $\tilde{S}_1(\tilde{\mathbf{x}},t)$ is a sphere parametrized by $S^2$ via the exponential map with basepoint $\tilde{x}\in \tilde{M}$, and that the measure on $S^2$ corresponding to $\tilde{\mu}_{t}$ is just the standard uniform probability measure $d\nu_{S^2}$. 
We let $\chi_{\tilde{O}_t}:S^2\to \{0,1\}$ be the characteristic function of the subset corresponding to $\tilde{O}_t$ and we compute for any large $N$:
$$c_1<\frac{1}{N}\sum_{k=1}^N \tilde{\mu}_{kt_0}(\tilde{O}_{kt_0}) = \frac{1}{N}\sum_{k=1}^N \int_{S^2} \chi_{\tilde{O}_{kt_0}} d\nu_{S^2} = \int_{S^2} \frac{1}{N}\sum_{k=1}^N \chi_{\tilde{O}_{kt_0}} d\nu_{S^2}$$
where the first inequality follows from (\ref{c10}).  So there are $c_2,c_3>0$, for any large $N$, on some subset of $S^2$ of $d\nu_{S^2}$-measure at least $c_3$,
$$\frac{1}{N}\sum_{k=1}^N \chi_{\tilde{O}_{kt_0}}> c_2$$
which is exactly the claim.
\end{proof}

\subsection{From equidistribution of geodesic spheres to intrinsic flat stability}
Let $(M,g_0)$ be a closed oriented hyperbolic manifold of dimension $n$.
One of the main technical tools in this section is the following volume entropy comparison, which roughly speaking says that if a sequence of metrics $g_i$ on $M$ approximates a metric space which is metrically dominated by $(M,g_0)$, then the volume entropy of $g_i$ is eventually strictly larger than that of $g_0$. Its proof relies on the equidistribution of geodesic spheres in hyperbolic manifolds, Theorem \ref{EM}.

\begin{theo} \label{cle}
Let $(M,g_0)$ be a closed oriented hyperbolic manifold of dimension $n\geq 2$. Suppose that the following holds:
\begin{enumerate}
\item 
there is a metric $d$ on $M$ such that there is a bi-Lipschitz bijection 
$$\Psi: (M,g_0) \to (M,d)$$ which is $1$-Lipschitz,
\item there are Riemannian metrics $g_i$ ($i\geq 1$) on $M$ so that for any $\varepsilon>0$, for all $i$ large enough, there are an open subset $A_i\subset M$, and a homotopy equivalence 
$$f_i: M\to M$$
such that 
the restriction $f_i : (A_i, g_i\vert_{A_i}) \to (M,d)$ is an $\varepsilon$-isometry. 
\end{enumerate}
Then, if $\Psi$ is not an isometry, we have
$$\liminf_{i\to \infty} h(g_i) >h(g_0).$$

\end{theo}

\begin{remarque}
We emphasize that $(A_i, g_i\vert_{A_i})$  denotes the metric space whose metric is induced by $g_i$ using paths in $A_i$. In particular, it is not in general isometric to $(A_i,\dist_{g_i} \vert_{A_i})$, where $\dist_{g_i} \vert_{A_i}$ is the restriction of $\dist_{g_i}$ to the subset $A_i$.
\end{remarque}

\begin{proof}

Consider small positive numbers $\eta, \varepsilon\in (0,1)$ to be fixed later, and consider $i$ large enough so that there is $A_i\subset M$ and a homotopy equivalence $f_i:M\to M$ whose restriction 
$$f_i : (A_i,g_i\vert_{A_i}) \to (M,d)$$
is an $\varepsilon$-isometry, as in condition (2). 


Let us then define ``$f_i$-lifts''. Given a point $p\in (M,g_0)$, we say that $p_i\in A_i$ is an $f_i$-lift of $p$ if $\Psi^{-1}(f_i(p_i)) $ is $\eta$-close to $p$ with respect to the hyperbolic metric $g_0$. 
Given a $g_0$-geodesic segment $\sigma$ in $(M,g_0)$ with endpoints $s,t$ (which is parametrized by arclength), we say that a curve $\sigma_i$ with endpoints $s_i,t_i$ in $(A_i,g_i\vert_{A_i})$ is a $f_i$-lift of $\sigma$ if 
\begin{itemize}
\item $s_i,t_i$ are $f_i$-lifts of $s,t$, 
\item $\length_{g_i}(\sigma_i)\leq (1+\eta)\length_{g_0}(\sigma),$
\item $\Psi^{-1}(f_i(\sigma_i)) \in [\sigma]\in \pi_1(M,B_{g_0}(s,\eta) \cup B_{g_0}(t,\eta) )$ where $B_{g_0}$ means $g_0$-geodesic ball.
\end{itemize}
By basic properties of $\varepsilon$-isometries \cite[Exercise 7.5.11]{BBI22} and since the bi-Lipschitz bijection $\Psi$ is $1$-Lipschitz, for any $\eta$, whenever $\varepsilon$ is small enough compared to $\eta$ and the injectivity radius of $(M,g_0)$, any $g_0$-geodesic segment $\sigma$ in $(M,g_0)$ admits an $f_i$-lift $\sigma_i$ in $ (A_i,g_i\vert_{A_i})$. 

Suppose now that the $1$-Lipschitz map $\Psi$ is not an isometry, which just means that there are two distinct points $x,y\in M$ so that 
\begin{equation}\label{longueur1}
d(\Psi(x),\Psi(y)) < \dist_{g_0}(x,y).
\end{equation}
Choose $\varepsilon, \eta$ and accordingly $i$, so that
\begin{equation}\label{longuueur2}
0< \varepsilon\ll \eta\ll \frac{\dist_{g_0}(x,y) -d(\Psi(x),\Psi(y))}{100}.
\end{equation}
Let $x_i,y_i$ be $f_i$-lifts of $x,y$. To fix ideas, let us assume without loss of generality that $\Psi^{-1}(f_i(x_i))=x$ and $\Psi^{-1}(f(y_i))=y$.
By the $\varepsilon$-isometry $f_i$ and (\ref{longuueur2}), 
\begin{equation}\label{longueur2}
\dist_{g_i}(x_i,y_i) \leq d(\Psi(x),\Psi(y)) +\varepsilon < \dist_{g_0}(x,y).
\end{equation}

Let $\sigma_i \subset (A_i,g_i\vert_{A_i})$ be a $g_i$-length minimizing segment which realizes the $g_i\vert_{A_i}$-distance between $x_i$ and $y_i$.
Consider the compact curve $\Psi^{-1}(f_i(\sigma_i))$ with endpoints $x,y$, and let us minimize its length among all  curves homotopic  to $\Psi^{-1}(f_i(\sigma_i))$ with same endpoints. This yields a $g_0$-geodesic segment 
$$\sigma : [0,\length_{g_0}(\sigma)]\to (M,g_0)$$ parametrized by arclength, with endpoints $x,y$. 
Note that, since $f_i$ is a homotopy equivalence, any $f_i$-lift of $\sigma$ with endpoints $x_i,y_i$ (that can always be ensured) is in fact homotopic (with fixed endpoints) to $\sigma_i$ inside $(M,g_i)$. 

By (\ref{longueur2}), by continuity and uniqueness properties for geodesic loops in hyperbolic manifolds, there are small disjoint open $g_0$-geodesic balls
$U_x,U_y$ containing respectively $x,y$ and some $\theta_0\in (0,1)$ with the following property: 
for any geodesic segment $\omega:[0,L]\to (M,g_0)$  such that $\omega(0) \in U_x$, $\omega(L) \in U_y$, and 
$$\omega \in [\sigma]\in \pi_1(M,U_x\cup U_y),$$
we can find a corresponding $f_i$-lift $\omega_i$ in $(M,g_i)$ and a curve $\hat{\omega}_i$ homotopic (with fixed endpoints) to $\omega_i$  such that 
\begin{equation}\label{rty2}
\length_{g_i}(\hat{\omega}_i) \leq \theta_0 \length_{g_0}(\omega).
\end{equation}
An important remark is that, since $\Psi$ is bi-Lipschitz, and since $f_i : (A_i,g_i\vert_{A_i}) \to (M,d)$ is an $\varepsilon$-isometry, the $g_0$-length of $\sigma$ is uniformly bounded independently of $i$. By compactness, we can assume without loss of generality that $\sigma$ is fixed and does not depend on $i$. For that reason, we will assume that $U_x,U_y,\theta_0$ only depend on $M,x,y,\sigma$ but not on $i$.
The notion of $f_i$-lifts of curves and their properties extend naturally to curves in the universal covers $(\tilde{M},g_0)$ and $(\tilde{M},g_i)$.

\vspace{1em}

Given a Riemannian metric $g$ on $M$ and a point $\mathbf{x}\in M$, let $\mathcal{L}_{\leq L}(g,\mathbf{x})$ be the collection of homotopy classes of loops with fixed basepoint $\mathbf{x}$ which contain at least one loop based at $\mathbf{x}\in M$ of $g$-length at most $L$. It is well-known that the volume entropy of $g$ is:
$$h(g) = \lim_{L\to \infty} \frac{\log(\mathbf{card} \mathcal{L}_{\leq L}(g, \mathbf{x}))}{L}$$
where $\mathbf{card}$ denotes the cardinality of a set. In particular, it does not depend on the choice of base point $\mathbf{x}$.

Fix a base point $\mathbf{x}\in (M,g_0)$ and a lift $\tilde{\mathbf{x}}\in \tilde{M}$ (here the ``lift''  belongs to the universal cover, it is not to be confused with the notion of $f_i$-lift).
By uniqueness of geodesic loops in homotopy classes of loops inside hyperbolic manifolds, we identify $\mathcal{L}_{\leq L}(g_0,\mathbf{x})$ with the set of geodesic loops based at $\mathbf{x}$ with length at most $L$. Classically, the volume entropy of the hyperbolic $n$-plane $(\tilde{M},g_0)$ is $n-1$, meaning that 
\begin{equation} \label{nmb}
\lim_{L\to \infty}\frac{\log(\mathbf{card} \mathcal{L}_{\leq L}(g_0,\mathbf{x}))}{L}=n-1. 
\end{equation}

The crux of the proof is that the equidistribution of lifts of geodesic spheres to the unit tangent bundle plus the distance comparison inequality (\ref{rty2}) force the volume entropy of $(M,g_i)$ to be strictly larger than $n-1$.

For all $i$ large, fix an $f_i$-lift $\mathbf{x}_i$ of the basepoint $\mathbf{x}$ inside $(M,g_i)$, and a lift $\tilde{\mathbf{x}}_i \in (\tilde{M},g_i)$ of $\mathbf{x}_i$ in the universal cover. 
As we saw earlier, we assume without loss of generality that $\sigma$ does not depend on $i$.
By Corollary \ref{coro equidistribution}, inequality (\ref{rty2}) and the properties of  $f_i$-lifts, we deduce that there are  some small $\theta_1\in (0,1)$ and $\varepsilon,\eta$ (this is where the latter are fixed) depending on $M,\sigma, U_x,U_y$ but independent of $i$, such that the following holds for all $i$ large. In the geodesic spheres $\tilde{S}(\tilde{\mathbf{x}},L)$ of universal cover $(\tilde{M},g_0)$, for any  $L$ large enough, there is an open subset $W_L \subset \tilde{S}(\tilde{\mathbf{x}},L)$ such that
$$\Area(W_L,g_0) \geq \theta_1 \Area(\tilde{S}(\tilde{\mathbf{x}},L),g_0),$$
and for any $q\in W_L$ and any $i$ large enough, the minimizing geodesic $l$ from $\tilde{\mathbf{x}}$ to $q$ admits an $f_i$-lift joining $\tilde{\mathbf{x}}_i$ to an $f_i$-lift of $q$ 
  in $(\tilde{M},g_i)$ which in turn is homotopic (with fixed endpoints) to a curve
 of $g_i$-length at most $(1-\theta_1) L$. In colloquial terms, a uniform fraction of points in $(\tilde{M},g_0)$ at $g_0$-distance $L$ from $\tilde{\mathbf{x}}$ admit $f_i$-lifts in $(\tilde{M},g_i)$ which are at $g_i$-distance significantly less than $L$ from $\tilde{\mathbf{x}}_i$.
 
 By basic hyperbolic geometry (volume of geodesic spheres and balls, etc.) and properties of $f_i$-lifts, the previous paragraph implies that for $i$ large enough, for all $L$ large enough there is a small $\theta_2\in (0,1)$ depending only on $M,\sigma, U_x,U_y$ such that for all $i,L$ large enough,
 \begin{itemize}
 \item
 there are distinct  points $p_1,...,p_K \in (\tilde{M},g_0)$ which are lifts  of $\mathbf{x}$ to $\tilde{M}$,
and their number satisfies 
 \begin{equation}\label{countt}
 K\geq \theta_2 \exp((1+\theta_2)(n-1)L),
 \end{equation}
 \item there are curves $ c_1,...,c_K\subset (\tilde{M},g_0)$ joining $\tilde{\mathbf{x}}$ to $p_1,...,p_K$  respectively,
and they admit $f_i$-lifts in $(\tilde{M},g_i)$, which are respectively homotopic (with fixed endpoints) to curves $c_{i,1},...,c_{i,K}\subset (\tilde{M},g_i)$ of $g_i$-lengths at most $L$,
\item each of the curves $c_{i,1},...,c_{i,K}$ joins $\tilde{\mathbf{x}}_i$ to some other lift of $\mathbf{x}_i$ inside the universal cover $(\tilde{M},g_i)$.
  \end{itemize}

We conclude from (\ref{nmb}) and (\ref{countt}) that for any $i$ large enough, for all $L$ large:
$$\log(\mathbf{card} \mathcal{L}_{\leq L}(g_i,{\mathbf{x}}_i)) \geq \log(\mathbf{card} \mathcal{L}_{\leq (1+\frac{\theta_2}{2})L}(g_0,\mathbf{x})).$$
In particular,
\begin{align*}
h(g_i) &\geq\liminf_{L\to \infty}  \frac{\log(\mathbf{card} \mathcal{L}_{\leq (1+\frac{\theta_2}{2})L}(g_0,\mathbf{x}))}{L}  = (1+\frac{\theta_2}{2} )h(g_0).
\end{align*}
Since $\theta_2 >0$ does not depend on $i$, the proof is complete.

\end{proof}

We are now ready to finish the proof of the intrinsic flat stability theorem.

\begin{theo}\label{intrinsic flat stability}
Let $(M,g_0)$ be a closed oriented hyperbolic manifold of dimension at least $3$. Let $\{g_i\}_{i\geq 1}$ be a sequence of Riemannian metrics on $M$ with $\Vol(M,g_i) = \Vol(M,g_0)$. If 
$$\lim_{i\to  \infty} h(g_i) = h(g_0) =n-1,$$
then there is a sequence of smooth subsets $Z_i \subset M$ such that 
$$\lim_{i\to \infty}\Vol(Z_i,g_i) =\lim_{i\to \infty}\Area(\partial Z_i,g_i) =0$$ and 
$(M\setminus Z_i,g_i\vert_{M\setminus Z_i})$ converges to $(M,g_0)$ in the intrinsic flat topology and Gromov-Hausdorff topology.
\end{theo}

\begin{proof}
Under the assumptions of the theorem, by combining Theorem \ref{technical uniqueness1} and Theorem \ref{cle}, we deduce that subsequentially, there are open subsets $A_i\subset M$ such that if
$$Z_i:=M\setminus A_i$$
then after renumbering,
\begin{itemize}
\item 
$\lim_{i\to \infty}\Vol(Z_i,g_i) =\lim_{i\to \infty}\Area(\partial Z_i,g_i) =0,$
\item 
$(M\setminus Z_i,g_i \vert_{M\setminus Z_i})$ converges in the intrinsic flat topology to an integral current space $C_\infty=(X_\infty,d_\infty,S_\infty),$
\item 
  $(M\setminus Z_i,g_i \vert_{M\setminus Z_i})$ converges to $(\spt S_\infty,d_\infty)$ in the Gromov-Hausdorff topology,
  \item
and  there is an isometric bijection
 $$\Psi : (M,g_0) \to (\spt S_\infty,d_\infty)$$
such that
 $$\Psi_\sharp(\llbracket 1_M\rrbracket) = S_\infty.$$
 \end{itemize}
In particular, $C_\infty$ is isomorphic as an integral current space to the hyperbolic manifold $(M,g_0)$. Since this integral current space is the only possible subsequential limit, there are $Z_i\subset M_i$ with  $\Vol(Z_i,g_i)$ and $\Area(\partial Z_i,g_i) $ converging to $0$, and $(M\setminus Z_i,g_i \vert_{M\setminus Z_i})$ converges to $(M,g_0)$ in the intrinsic flat and Gromov-Hausdorff topologies (without the need to take subsequences).
\end{proof}

Recall that intrinsic flat convergence implies weak convergence (see Section \ref{appendix a}). Given a Riemannian metric $g$ on $M$, the mass measure of the integral current space $(M,\dist_{g},\llbracket 1_M\rrbracket)$ is equal to the usual volume measure $\dvol_g$ on $M$. 
The proof of Theorem \ref{theorem:stable} is then completed by combining Theorem   \ref{intrinsic flat stability} and the following general lemma proved by Portegies, which yields that weak convergence plus volume convergence implies Gromov-Prokhorov convergence for Riemanian manifolds:
\begin{lemme} \cite[Lemma 2.1]{Portegies15}
Suppose $Z$ is a complete metric space, and $\{T_i\}$ is a sequence of integral currents in $Z$ converging weakly to an integral current $T$. 
Moreover, assume that $\mathbf{M}(T_i)$ converges to $\mathbf{M}(T)$. 
Then the mass measure $\|T_i\|$ converges weakly to $\|T\|$ as measures on $Z$.
\end{lemme}

\begin{remarque} \label{optimal}

Sometimes, as in Theorem \ref{theorem:stable},
a sequence of $n$-manifolds $(M_i,g_i)$ converges  to a nice limit space $X$  in a given canonical topology $\mathscr{T}$ only after removing negligible subsets  $Z_i$ from $M_i$. For an example different from Theorem \ref{theorem:stable} and related to scalar curvature, see \cite{DS23}. 
To quantify that phenomenon, we can look at the coarse dimension of $\partial Z_i$.
To measure the coarse dimension of a manifold $(N,h)$, we propose the following notion of ``Euclidean $q$-area'' $\mathcal{A}_{q}(N,h)$:
 $$\mathcal{A}_{q}(N,h):= \sup \{\mathcal{H}^q(\pi(N)); \quad \pi:(N,h) \to \mathbb{R}^q \text{ is a $1$-Lipschitz map}\}$$
where $\mathcal{H}^q$ denotes the standard $q$-dimensional Hausdorff measure. 
Let us declare that $(\partial Z_i,g_i)$  has  coarse dimension $q-1$ if  $\lim_{i\to \infty} \mathcal{A}_{q}(\partial Z_i,g_i)=0$.
\footnote{A reason why we do not use the notion of Uryson width instead of Euclidean $q$-area is that $Z_i$ can usually be chosen to have small $0$-dimensional Uryson width.}
As a corollary of Theorem \ref{theorem:stable}, for the volume entropy inequality, hyperbolic manifolds of dimension $n\geq 3$ are ``codimension 2 stable'' in the measured Gromov-Hausdorff topology. This is in general optimal. What about other stability and convergence problems?

\end{remarque}

\bibliographystyle{alpha}
\bibliography{biblio_22_10_11}

\end{document}